\documentclass[10pt,a4paper,english,twoside]{article}

\usepackage{authblk}
\usepackage[style=alphabetic,backend=bibtex,firstinits=true]{biblatex}
\usepackage[a4paper]{geometry}
\usepackage{fancyhdr, lipsum}
\usepackage[displaymath,mathlines,running,pagewise]{lineno}

\usepackage{amsfonts,stmaryrd,amsmath,amsthm,amssymb}
\usepackage{graphicx,color}
\usepackage{multimedia}
\usepackage{float}
\usepackage{placeins}

\usepackage{mathtools}
\usepackage{mathrsfs}

\newtheorem{theorem}{Theorem}
\newtheorem{lemma}{Lemma}
\newtheorem{remark}{Remark}
\newtheorem{proposition}{Proposition}
\newtheorem{definition}{Definition}

\newcommand \bmu {\boldsymbol{\mathrm{\mu}}}
\newcommand \bW {\mathbf{W}}

\newcommand \p {\partial}

\newcommand \R {\mathbb{R}}

\renewcommand \L {\mathrm{L}}
\newcommand \W {\mathrm{W}}

\newcommand \LL {\mathrm{L}}
\newcommand \WW {\mathrm{W}}
\newcommand \BB {\mathrm{B}}

\newcommand \I {\mathrm{I}}
\newcommand \Id {\mathrm{Id}}

\renewcommand \d {\mathrm{d}}

\begingroup\makeatletter\ifx\SetFigFont\undefined%
\gdef\SetFigFont#1#2#3#4#5{%
  \reset@font\fontsize{#1}{#2pt}%
  \fontfamily{#3}\fontseries{#4}\fontshape{#5}%
  \selectfont}%
\fi\endgroup%

\title{Design of the monodomain model by artificial neural networks\thanks{This work is supported by...}}

\author[1, 2]{S\'ebastien Court}
\affil[1]{\begin{small}Department of Mathematics, University of Innsbruck, Technikerstr. 13, A-6020 Innsbruck, Austria.\end{small}}
\affil[2]{\begin{small}Digital Science Center, University of Innsbruck, Innrain 15, A-6020 Innsbruck, Austria. Email: {\tt sebastien.court@uibk.ac.at}\end{small}}
\author[3, 4]{Karl Kunisch}
\affil[3]{\begin{small}Institute for Mathematics and Scientific Computing, University of Graz, Heinrichstr. 36, A-8010 Graz, Austria.\end{small}}
\affil[4]{\begin{small}Radon Institute, Austrian Academy of Sciences, Altenbergstr. 69, A-4040 Linz, Austria. Email: {\tt karl.kunisch@uni-graz.at}.\end{small}}

\begin{document}

\maketitle

\begin{abstract}
We propose an optimal control approach in order to identify the nonlinearity in the monodomain model, from given data. This data-driven approach gives an answer to the problem of selecting the model when studying phenomena related to cardiac electrophysiology. Instead of determining coefficients of a prescribed model (like the FitzHugh-Nagumo model for instance) from empirical observations, we design the model itself, in the form of an artificial neural network. The relevance of this approach relies on the \textcolor{black}{approximation capacities} of neural networks. We formulate this inverse problem as an optimal control problem, and provide mathematical analysis and derivation of optimality conditions. One of the difficulties comes from the lack of smoothness of activation functions which are classically used for training neural networks. Numerical simulations demonstrate the feasibility of the strategy proposed in this work.
\end{abstract}

\hfill 

\tableofcontents

\section{Introduction}
The present paper addresses the question of determining the nonlinear terms for the monodomain model, with a data-driven approach. We could identify the parameters for a given class of models, like for the \textcolor{black}{FitzHugh-Nagumo model~\cite{FITZHUGH1961}, for which the coefficients need to be estimated~\cite{Yanqiu2012, Doruk2019}. In general this is achieved by linear regression, by reducing the difference between the outcomes of the model and the empirical observations}. \textcolor{black}{Recently, the {\it SINDy} algorithm~\cite{Brunton1, Brunton2, Brunton3, Messenger2021} proposed a direct and efficient approach for identifying nonlinear dynamics with linear regression techniques, from data measurements. The approach of SINDy relies on a linear combination of nonlinear terms stored in a dictionary. While the underlying nonlinear terms of an equation are certainly and precisely (re-)discovered with SINDy, as far as they are represented in the dictionary, it is not yet clear how such an approach could help when trying to identify nonlinear terms that are not represented in the considered dictionary. Further, modeling terms of equations with the traditional tools and intuition of a physicist can entail some limits. In this case neural network techniques can provide a viable alternative.}

\textcolor{black}{In order to tackle the nonlinearities, the modern Koopman theory~\cite{Koopman} proposes to learn change of coordinates that reduces a nonlinear phenomena to a linear differential equation, for which standard linear algebra tools enable us to determine the dynamics. Instead of pursuing the goal of identifying nonlinear dynamics by deploying linear algebra related tools, we propose in the present paper to change and enrich the approximation means by adopting purely nonlinear tools, and designing the nonlinear terms of a differential equation in the form of {\it neural networks}.}

For justifying this approach, we rely on the approximation capabilities of neural networks. First, the universal approximation theorem~\cite{Cybenko89} shows the possibility of approximating any continuous function on a finite-dimensional compact set by a neural network with a single hidden layer. Next these approximation capabilities were extensively developed, like in~\cite{Hornik92, Pinkus93}. More recently, Grohs \& al. \cite{Grohs2019} gave bounds on the width and depth of a neural network approximating a polynomial function with a given precision. Note that strong nonlinearities -- of exponential type for instance -- can be satisfactorily approximated by polynomial functions. The specific interest of neural networks, compared with standard mode decomposition, lies in the successive composition action of the input with {\it activation functions}, leading to exponential convergence (see~\cite[Proposition~III.3]{Grohs2019}) when raising the number of successive compositions -- namely the number of layers. The price to pay lies in the determination of coefficients, the so-called {\it weights} of the neural network, that appear \textcolor{black}{unavoidably nonlinear} through the successive compositions.


A data-driven approach for designing the nonlinear dynamics of the monodomain model can be formulated as follows:
\begin{eqnarray*}
\left\{ \begin{array} {l}
\displaystyle \min_{\mathbf{W}} \sum_{k=1}^K  \left\| \mathrm{z} - \mathrm{z}_{\mathrm{data},k} \right\|^2 \\[10pt]
\text{\textcolor{black}{where $\mathbf{W}$ are affine maps that parameterize the nonlinear terms of~\eqref{sysmain} below.}}
\end{array}\right.
\end{eqnarray*}
The variable $\mathrm{z} = (v,w)$ denotes the unknown of the state system~\eqref{sysmain} with $(v_0,w_0) = \mathrm{z}^{(0)}_{\mathrm{data},k} \in \R^2$ as initial conditions, and $\mathrm{z}_{\mathrm{data},k}$ represent measurements of the state. The model is given by
\begin{linenomath}
\begin{subequations} \label{sysmain}
\begin{eqnarray}
\displaystyle \frac{\p v}{\p t} - \nu\Delta v + \phi_v(v,w) = f_v & &
\text{in } \Omega\times (0,T), \label{sysmain-1} \\
\displaystyle \frac{\p w}{\p t} + \delta w + \phi_w(v,w) = f_w & &
\text{in } \Omega\times (0,T),\\
\displaystyle \frac{\p v}{\p n} = 0 & & \text{on } \p \Omega \times (0,T),\\
(v,w)(\, \cdot \, ,0) = (v_0,w_0) & & \text{in } \Omega.
\end{eqnarray}
\end{subequations}
\end{linenomath}
The domain $\Omega \subset \R^d$ is assumed to be bounded and smooth. The constant parameters $\delta >0$ and $\nu>0$ are given. The nonlinear functions $(v,w) \mapsto \phi_v(v,w)$ and $(v,w) \mapsto \phi_w(v,w)$ are components of a function $\Phi = (\phi_v,\phi_w)$, chosen in the form of a {\it neural network}, namely a function whose expression is given by
\begin{equation}
\Phi(\mathrm{z})  =  \text{\textcolor{black}{$
\left( \begin{matrix} \displaystyle
\phi_v(\mathrm{z}) \\ \phi_w(\mathrm{z})
\end{matrix} \right) = $
}}
\left\{ \begin{array} {ll}
W_2(\rho (W_{1}(\mathrm{z})))
& \text{if } L=2 , \\[5pt]
W_L(\rho (W_{L-1}(\rho(\dots \rho(W_1(\mathrm{z})))))) & \text{if } L \geq 3,
\end{array} \right. \label{sysmain2}
\end{equation}
for $\mathrm{z} = (v,w) \in \R^2$. The parameters $W_\ell$ correspond to the so-called {\it weights} of the neural networks. \textcolor{black}{These are affine functions between Euclidean spaces,} and are the control functions to be optimized. The activation function denoted by $\rho$ will be assumed to be globally Lipschitz. \textcolor{black}{For example, ReLU functions and their smooth approximations such as GELU, SiLU or Softplus would correspond to our assumption. Other smooth examples such as the Logistic activation function, hyperbolic tangent, ELU, or non-smooth examples such as Leaky ReLU, PreLU, would also fit our framework. We refer to~\cite[section~6.3.1]{Goodfellow2016} for further comments.} The results obtained in this paper are maybe also true for activation functions which are only locally Lipschitz, but we prefer to sacrifice accuracy for simplicity.

Following~\cite[Proposition~III.3]{Grohs2019}, we note that by means of artificial neural networks we can approximate the FitzHugh-Nagumo model, which is a polynomial representation of the dynamics in the following form:
\begin{eqnarray*}
 \displaystyle \frac{\p v}{\p t} - \nu\Delta v + av^3+bv^2+cv +dw = f_v & &
 \text{in } \Omega\times (0,T),\\
 \displaystyle \frac{\p w}{\p t} + \eta w + \gamma v = f_w & &
 \text{in } \Omega\times (0,T),\\
\displaystyle \frac{\p v}{\p n} = 0 & & \text{on } \p \Omega \times (0,T),\\
 (v,w)(\, \cdot \, ,0) = (v_0,w_0) & & \text{in } \Omega.
\end{eqnarray*}
As mentioned in the original article~\cite{FITZHUGH1961}, this is the simplest representation which models, in terms of equations, experimental observations of behaviors like the presence of three equilibrium states, and the harmonic oscillator nature of the dynamics. The coefficients $a$, $b$, $c$, $d$, $\eta$ and $\gamma$ have to be identified empirically, in order to minimize the bias between the output of the model and the experimental data. This approximation corresponds to a basic mode decomposition of the nonlinear underlying dynamics, and in spite of its apparent simplicity, the investigation of related theoretical questions as well as numerical simulations of such a model is already a non-trivial task. We propose here to to change the nature of the approximation, and this leads us to other difficulties. In this fashion, the recent contribution~\cite{Fresca-1}, using Deep Learning techniques for modeling the monodomain equations with a data-driven approach in the context of cardiac electrophysiology, proposed a similar approach with focus on numerical realization.

One could prove wellposedness for system~\eqref{sysmain} in the framework of strong solutions, as it is done for the FitzHugh-Nagumo model in~\cite{Breiten1, Breiten2} in the Hilbert space case, or in~\cite{Hieber2018, Hieber2020-1, Hieber2020-2} in the general context of strong $\L^p$-maximal regularity. Given the way we consider the nonlinearity, namely designed with artificial neural networks, sticking to the functional setting of these previous works would lead to require invariance of functional spaces under the action of the activation functions of the neural network, and this would lead us to assume that they are at least of class $C^1$ (see~\cite[Lemma~A.2.]{BB74}), which is too restrictive in view of the performances of ReLU functions for training neural networks (see~\cite[p.~438]{LeCun2015}). The question of weak solutions was addressed in the Hilbert space case in~\cite{Wagner1, Wagner2}, for a given class of nonlinearities. In the present paper we will consider solutions which have the weak maximal parabolic regularity, namely
\begin{eqnarray*}
v \in \L^s(0,T; \WW^{1,p}(\Omega)) \cap \W^{1,s}(0,T; \WW^{-1,p}(\Omega)).
\end{eqnarray*}
for $p<d$ and $s>1$. In the linear case, existence results in this fashion were obtained in~\cite{Rehberg1}. They rely on resolvent estimates for parabolic operators in $\WW^{-1,p}(\Omega)$, obtained in~\cite{Groeger1, Groeger2, Groeger3}.

For studying the optimal control problem, the non-smoothness of the activation functions leads us to consider regularization techniques recently used in similar contexts (see~\cite{Meyer1, Meyer2}), namely non-smooth parabolic equations. In~\cite{Meyer1, Meyer2}, the emphasis is on the properties and consequences of the regularization. In our contribution, the non-smoothness of the activation functions does not allow us to define an adjoint system -- in the classical sense -- for the non-regularized system, and we prefer to focus on the specific structure of the optimal control problem, namely a control operator designed as a neural network, involving weights as control functions in the successive compositions by activation functions.

The plan is organized as follows: The functional framework for the state variables and the formulas concerning neural networks are given in section~\ref{sec-func}. Wellposedness questions are addressed in section~\ref{sec-para}, in the context of parabolic maximal regularity. An optimal control problem is introduced in section~\ref{sec-opt} for identifying the nonlinearities of the monodomain model, for which we prove existence of minimizers. Next, optimality conditions are derived in section~\ref{sec-reg} in the case of smooth activation functions. Section~\ref{sec-passing} is devoted to the transposition of the optimality conditions in the case of non-smooth activation functions. In section~\ref{sec-num} we propose numerical tests showing feasibility and relevance of our approach.
Conclusions and comments are given in section~\ref{sec-conc}.

\section{Notation and preliminaries} \label{sec-func}

Let be $T>0$ and $\Omega$ a smooth bounded domain of $\R^d$, with $d = 2$ or $3$. For any index $\kappa >1$, we will denote by $\kappa'$ its {\it dual exponent}, such that $1/{\kappa'} + 1/\kappa = 1$, and keep in mind H\"{o}lder's inequality, which will be used throughout the paper.

\subsection{Negative Besov and Sobolev spaces}

The Sobolev spaces of type $\WW^{s,p}(\Omega)$ incorporating the homogeneous Dirichlet condition are defined as in~\cite[section~3]{Amann2004}, namely
\begin{linenomath}\begin{equation*}
\tilde{\WW}^{s,p}(\Omega) := \left\{ \begin{array} {ccl}
\left\{ u \in \WW^{s,p}(\Omega)\mid \gamma u = 0 \right\} & & \text{if } 1/p < s \leq 2, \\
\WW^{s,p}(\Omega) & & \text{if } 0 \leq s < 1/p, \\
\left(\tilde{\WW}^{-s,p'}(\Omega)\right)' & & \text{if } -2 \leq s < 0, \ s \notin \mathbb{Z} + 1/p,
\end{array}\right.
\end{equation*}\end{linenomath}
where $\gamma$ denotes the trace operator on $\p \Omega$. For the sake of simplicity, in the rest we will denote $\WW^{s,p}(\Omega) = \tilde{\WW}^{s,p}(\Omega)$. Next, when $s>0$, we denote the Besov spaces, defined by real interpolation as follows:
\begin{linenomath}\begin{equation}
\BB^{s}_{p,q}(\Omega) = \left(\LL^p(\Omega) ; \WW^{m(s),p}(\Omega) \right)_{s/m(s), q}, \label{Besov-inter}
\end{equation}\end{linenomath}
where $m(s) := \lceil s \rceil$ is the smallest integer non-smaller than $s$. When $s<0$, we define $\BB^{s}_{p,q}(\Omega) = \left(\BB^{-s}_{p',q'}(\Omega) \right)'$, which is coherent with the following identity given in~\cite[Theorem~3.7.1, p.~54]{Bergh}:
\begin{linenomath}\begin{equation*}
\left( \BB^{-s}_{p',q'}(\Omega) \right)' = \left( (\LL^{p'}(\Omega); \WW^{m(-s),p'})_{-s/m(-s),q'} \right)'
= (\LL^{p}(\Omega); \WW^{-m(-s), p}(\Omega))_{-s/m(-s),q}.
\end{equation*}\end{linenomath}
Thus, when $s<0$, we can set $m(s) = -m(-s)$, and then  equality~\eqref{Besov-inter} holds for $s<0$ too.

\subsection{Functional setting}
For $s>1$ and $p>1$, we set
\begin{linenomath}\begin{equation*}
X_{s,p}^T :=  \L^s(0,T; \WW^{1,p}_{\#}(\Omega)) \cap \W^{1,s}(0,T; \WW^{-1,p}(\Omega)),
\text{ where } \WW^{1,p}_{\#}(\Omega):= \left\{v \in \WW^{1,p}(\Omega) \mid \int_{\Omega}v \, \d \Omega = 0 \right\}.
\end{equation*}\end{linenomath}
Observe that the interpolation space $\left(\WW^{-1,p}(\Omega), \WW^{1,p}(\Omega)\right)_{1-1/s,s}$ coincides with the Besov space $\BB^{1-2/s}_{p,s}(\Omega)$ (see~\cite{Triebel}). So the following continuous embedding holds (see~\cite{Amann-book}, Theorem~III.4.10.2 page~180):
\begin{linenomath}\begin{equation}
X_{s,p}^T  \hookrightarrow
C\left([0,T];  \BB^{1-2/s}_{p,s}(\Omega)\right).
\label{super-embedding}
\end{equation}\end{linenomath}
We will assume throughout the paper that
\begin{linenomath}\begin{equation} \label{ass-ps}\tag{C1}
1<s, \quad 1< p < d.
\end{equation}\end{linenomath}
Let $r$ denote any index such that
\begin{linenomath}\begin{equation} \label{ass-r}\tag{C2}
(dp)/(d+p) < r \leq p.
\end{equation}\end{linenomath}
Assumption~\eqref{ass-r} implies that
\begin{linenomath}\begin{equation} \label{ass-r2}\tag{C2'}
(dp)/(d+p) < r < (dp)/(d-p).
\end{equation}\end{linenomath}
In this case, from the Rellich-Kondrachov theorem, the continuous embeddings $\WW^{1,p}(\Omega) \hookrightarrow \LL^r(\Omega) \hookrightarrow \WW^{-1,p}(\Omega)$ hold and are compact. We will denote by
\begin{linenomath}\begin{equation*}
Y_{s,r}^T :=  \W^{1,s}(0,T; \LL^r(\Omega))
\end{equation*}\end{linenomath}
the space endowed with the norm given by
\begin{linenomath}\begin{equation*}
\text{\textcolor{black}{$\| w\|_{Y_{s,r}^T}$}} :=
\|w\|_{\L^\infty(0,T;\LL^r(\Omega))} + \|w\|_{\W^{1,s}(0,T;\LL^r(\Omega))}.
\end{equation*}\end{linenomath}
Let us define what we call a {\it weak solution} of system~\eqref{sysmain}-\eqref{sysmain2}.

\begin{definition} \label{def-weak}
We say that $(v,w) \in X_{s,p}^T \times Y_{s,r}^T$ is a weak solution of system~\eqref{sysmain}-\eqref{sysmain2} if for all $(\varphi_v,\varphi_w) \in \L^{s'}(0,T;\WW^{1,p'}(\Omega)) \times \L^{s'}(0,T;\LL^{r'}(\Omega))$ we have
\begin{eqnarray*}
& & \int_0^T\left(\left\langle \frac{\p v}{\p t}-f_v \, ; \, \varphi_v \right\rangle_{\WW^{-1,p}(\Omega);\WW^{1,p'}(\Omega)}
+ \nu \left\langle \nabla v \, ; \,  \nabla \varphi_v\right\rangle_{\LL^{p}(\Omega);\LL^{p'}(\Omega)}
+ \left\langle \phi_v(v,w) \, ; \, \varphi_v \right\rangle_{\LL^{r}(\Omega);\LL^{r'}(\Omega)}
\right) \d t = 0, \\
& & \int_0^T  \left\langle \frac{\p w}{\p t} + \delta w + \phi_w(v,w) - f_w \, ; \,  \varphi_w \right\rangle_{\LL^{r}(\Omega);\LL^{r'}(\Omega)} \d t  = 0, \\
& & v(x,0) = v_0(x), \quad w(x,0) = w_0(x), \quad \text{for a.e. } x\in \Omega.
\end{eqnarray*}
\end{definition}
In the definition above, the following regularity properties are supposed to hold:
\begin{eqnarray*}
\phi_v(v,w),  \, \phi_w(v,w)  \in  \L^s(0,T;\LL^r(\Omega)),
& & \forall (v,w) \in X^T_{s,p} \times Y^T_{s,r}.
\end{eqnarray*}
For $\phi_v$ and $\phi_w$ given as in~\eqref{sysmain2}, this regularity is proved in Lemma~\ref{lemma-nem}.

\subsection{Properties of neural networks} \label{subsec-NN}
For an activation function $\rho$, assumed to be globally Lipschitz, we consider a neural network with $L$ layers:
\begin{eqnarray*}
\Phi_L(\mathrm{z}, W_1,\dots, W_L) & = &
\left\{ \begin{array} {ll}
W_2(\rho (W_{1}(\mathrm{z})))
& \text{if } L=2 , \\[5pt]
W_L(\rho (W_{L-1}(\rho(\dots \rho(W_1(\mathrm{z})))))) & \text{if } L \geq 3.
\end{array} \right.
\end{eqnarray*}
The weights $(W_\ell)$ are affine functions from $\R^{n_{\ell-1}}$ to $\R^{n_{\ell}}$, of the form $W_\ell (\mathrm{z})= A_\ell\mathrm{z} + b_\ell$ with $A_\ell \in \R^{n_{\ell}\times n_{\ell-1}}$ and $b_\ell \in \R^{n_{\ell}}$. The resulting affine spaces are denoted by $\mathrm{Aff}(n_{\ell-1}, n_{\ell},\R)$. Naturally we impose $n_0 = n_L = 2$. We denote by
\begin{linenomath}\begin{equation*}
\mathbf{W} \in \mathcal{W} :=
\prod_{\ell = 1}^{L}\mathrm{Aff}(n_{\ell-1}, n_{\ell},\R)
\end{equation*}\end{linenomath}
the family $(W_\ell)_{1\leq \ell \leq L}$. The space $\mathcal{W}$ is endowed with the Euclidean norm given by
\begin{eqnarray*}
\| \mathbf{W} \|^2_{\mathcal{W}} := \sum_{\ell=1}^L
\left(|A_{\ell}|^2_{\R^{n_{\ell}\times n_{\ell-1}}} + |b_\ell|^2_{\R^{n_{\ell}}} \right)
=  \sum_{\ell=1}^L
\left(\mathrm{trace}(A_{\ell}^TA_{\ell}) + |b_\ell|^2_{\R^{n_{\ell}}} \right) .
\end{eqnarray*}
Since the functions $W_\ell$ have values in $\R^{n_{\ell}}$, the compositions of type $\rho(W_\ell(\mathrm{z}))$, for $\mathrm{z} \in \R^{n_{\ell-1}}$ have to be understood coordinate-wise as $\rho((W_\ell(\mathrm{z}))_i)$, for $i \in \{1,\dots, n_{\ell}\}$.\\

Observe that for all $1 \leq \ell \leq L-1$ the following formula holds:
\begin{eqnarray}
\Phi_{\ell+1}(\mathrm{z}, W_1,\dots, W_{\ell+1}) & = & W_{\ell+1} \circ\rho\big(\Phi_{\ell}(\mathrm{z}, W_1,\dots, W_{\ell})\big). \label{formulaPhi}
\end{eqnarray}
We assume that the activation function $\rho$ is globally Lipschitz on $\R$. Thus $\rho'$ exists in $\L^\infty(\R)$. 
From~\eqref{formulaPhi}, by induction we can verify that the derivative of $\Phi_L$ with respect to $\mathrm{z}$ satisfies the identity
\begin{linenomath}\begin{equation}
\nabla_{\mathrm{z}}\Phi_L(\mathrm{z}, W_1,\dots, W_L)  =  A_L\mathrlap{\prod_{\ell = 1}^{L-1}}{\hspace*{-0.2pt}\longrightarrow} \rho'\big(\Phi_\ell(\mathrm{z},W_1,\dots, W_{\ell})\big) A_{\ell} . \label{grad_nn}
\end{equation}\end{linenomath}
The symbol $\displaystyle \mathrlap{\prod_{\ell = 1}^{L-1}}{\hspace*{-0.2pt}\longrightarrow}$ indicates that the product is calculated iteratively by multiplying to the left the different matrices, when the index $\ell$ increases: $\displaystyle \mathrlap{\prod_{\ell = 1}^{L}}{\hspace*{-1pt}\longrightarrow} B_{\ell} = B_L B_{L-1}\dots B_2 B_1$.
Here again the terms of type $\rho'\big(\Phi_\ell \big)$ are vectors of $\R^{n_{\ell}}$ that have to be understood coordinate-wise. The vector/matrix product $\rho'\big(\Phi_\ell(\mathrm{z},W_1,\dots, W_{\ell})\big) \, A_{\ell}$ is a matrix, and has to be understood as $(\rho' \, A)_{ij}  =  \rho'_i \, A_{ij}$. Then, from the mean-value theorem, we get the following Lipschitz estimate:
\begin{eqnarray}
|\Phi_L(\mathrm{z}_1,\bW) - \Phi_L(\mathrm{z}_2,\bW)|_{\R^2} & \leq &
\left(\mathrlap{\prod_{\ell = 1}^{L}}{\hspace*{-1pt}\longrightarrow} |A_\ell |_{\R^{n_{\ell}\times n_{\ell-1}}} \right)
\|\rho'\|_{\infty}^{\mathcal{L}_{\mathcal{W}}} | \mathrm{z}_1-\mathrm{z}_2 |_{\R^2}, \label{resLip}
\end{eqnarray}
where we introduced $\mathcal{L}_{\mathcal{W}} := \displaystyle \sum_{\ell = 1}^{L-1} n_{\ell}$. The derivative of $\Phi_L$ with respect to the weights $W_\ell$ is given by the following formula, for $1 \leq \ell \leq L-1$, for all $\tilde{W} : \mathrm{z} \mapsto \tilde{A}\mathrm{z}+\tilde{b}$:
\begin{eqnarray}
\begin{array} {rcl}
\displaystyle \frac{\p \Phi_L}{\p W_\ell}(\mathrm{z}, \bW).\tilde{W} & = &
\displaystyle \left(\mathrlap{\prod_{k=\ell}^{L-1}}{\hspace*{-0.2pt}\longrightarrow} A_{k+1}\rho'\big(\Phi_k(\mathrm{z}, W_1,\dots, W_k)\big)\right) \left( \tilde{W} \circ  \rho\big(\Phi_{\ell-1}(\mathrm{z}, W_1, \dots, W_{\ell-1})\big)\right), \\
\displaystyle\frac{\p \Phi_L}{\p W_L}(\mathrm{z}, \bW).\tilde{W} & = &
\displaystyle \left( \tilde{W} \circ  \rho\big(\Phi_{L-1}(\mathrm{z}, W_1, \dots, W_{L-1})\big)\right).
\end{array} \label{formula_phiW}
\end{eqnarray}
Here the products have to be understood as follows:
\begin{eqnarray*}
\left(\frac{\p \Phi_L}{\p W_\ell}(\mathrm{z},\mathbf{W}).\tilde{W}\right)_i & = &
 \left(\mathrlap{\prod_{k=\ell}^{L-1}}{\hspace*{-0.2pt}\longrightarrow} \left(A_{k+1}\right)_{ij}\rho'\big((\Phi_k(\mathrm{z},W_1,\dots,W_k))_j\big)\right) \left( \tilde{W} \rho\big(\Phi_{\ell-1}(\mathrm{z},W_1,\dots, W_{\ell-1})\big)\right)_j.
\end{eqnarray*}
Throughout the paper we will omit the index $L$, and will denote by $\displaystyle \nabla_{\bW} \Phi$ the vector whose components are the functions $\displaystyle \frac{\p \Phi}{\p W_{\ell}}$. Note that the latter are made of two main components, namely the sensitivity with respect to~$A_{\ell}$, and the sensitivity with respect to~$b_{\ell}$.

\section{Parabolic maximal regularity} \label{sec-para}

\subsection{Weak maximal parabolic regularity for the Laplace operator} \label{sec-maxparaDisser}
As in~\cite[Definition~5.6]{Rehberg2}, we define $\mathcal{A}_p: \WW^{1,p}_{\#}(\Omega) \rightarrow \WW^{-1,p}(\Omega)$ by
\begin{eqnarray*}
\langle \mathcal{A}_p v \, ;\varphi \rangle_{\WW^{-1,p}(\Omega); \WW^{1,p'}(\Omega)} := \nu \int_{\Omega} \nabla v \cdot \nabla \varphi \, \d \Omega.
\end{eqnarray*}
We first recall existing results providing conditions under which the Laplace operator owns the parabolic maximal regularity in the space~$X^T_{s,p}$. From~\cite[Example~6.6]{Rehberg2}, we recall that $\mathcal{A}_p + \I: \WW^{1,p}_{\#} \rightarrow \WW^{-1,p}$ is an isomorphism for all . Further, from the Poincar\'e-Wirtinger inequality, there exists $C>0$ such that
\begin{linenomath}\begin{equation*}
\|(\mathcal{A}_p + \I)v\|_{\WW^{-1,p}(\Omega)} = \| v \|_{\WW^{1,p}(\Omega)} \leq
C \| \nabla v \|_{\LL^p(\Omega)}
= C\|\mathcal{A}_pv \|_{\WW^{-1,p}},
\quad \forall v\in \WW^{1,p}_{\#}(\Omega).
\end{equation*}\end{linenomath}
Thus $\mathcal{A}_p : \WW^{1,p}_{\#}(\Omega) \rightarrow \WW^{-1,p}(\Omega)$ is also an isomorphism for all  $p\in(1,\infty)$. Following~\cite[Corollary~6.12]{Rehberg2} \textcolor{black}{that addresses a more general context}, we \textcolor{black}{know} that there exists $\delta \in (0,1)$ such that $\mathcal{A}_p : \WW^{1,p}_{\#}(\Omega) \rightarrow \WW^{-1,p}(\Omega)$ satisfies maximal parabolic regularity for all $p\in (2-\delta, \infty)$ (see also~\cite{Groeger1} as a first contribution).
We will then consider exponents $p$ such that the following additional assumption is fulfilled:
\begin{linenomath}\begin{equation} \label{ass-max}\tag{C3}
\text{The operator $\mathcal{A}_p$ satisfies the maximal parabolic $\L^s(0,T;\WW^{-1,p}(\Omega))$-regularity.}
\end{equation}\end{linenomath}
The notion of \textcolor{black}{maximal} parabolic regularity is understood in the sense of~\cite[Definition~2.2]{Rehberg2}, \textcolor{black}{namely that for any $f \in \L^s(0,T;\W^{-1,p}(\Omega))$ there exists a unique $v \in X_{s,p}^T$ satisfying $\dot{v} + \mathcal{A}_p v = f$ for almost all $t \in(0,T)$, with the initial condition $v(0) = 0$}. Recall that this property \textcolor{black}{is independent of $T \in (0,\infty)$, and also of the exponent $s \in (1,\infty)$. It  is also transferable to the case where $v(0) = v_0 \neq 0$, as far as $v_0$ lies in the trace space of $X_{s,p}^T$, namely~$\BB^{1-2/s}_{p,s}(\Omega)$. We refer to~\cite{Hieber2003} for further aspects of the parabolic theory.}

As far as we know, existence of weak solutions for the \textcolor{black}{coupled} system~\eqref{sysmain} in the general context of parabolic maximal regularity has not yet been investigated. However, the semilinear equation~\eqref{sysmain-1} has been treated in~\cite{Amann2004} in this functional framework. We refer to section~\ref{secthAmann} for more detailed comments. But the existing results do not apply {\it a priori} directly to the coupled system~\eqref{sysmain}, and so in this section we prove a local-in-time existence result, completed by sufficient conditions leading to global existence of solutions. The strategy for obtaining local-in-time existence is similar to~\cite{Amann2004}, namely the utilization of the Banach fixed-point theorem.

\subsection{The linear system}

Let be $T >0$. In this subsection we are interested in the following linear system
\begin{eqnarray}
\left\{ \begin{array} {ll}
\displaystyle \frac{\p v}{\p t} - \nu \Delta v = F_v &
\text{in } \Omega\times (0,T),\\[10pt]
\displaystyle \frac{\p w}{\p t} + \delta w  = F_w &
\text{in } \Omega\times (0,T),\\[10pt]
\displaystyle \frac{\p v}{\p n} = 0 & \text{on } \p \Omega \times (0,T),\\[10pt]
(v,w)(\cdot,0) = (v_0,w_0) & \text{in } \Omega.
\end{array} \right. \label{syslin}
\end{eqnarray}
We assume that the data satisfy
\begin{eqnarray*}
F_v\in \L^s(0,T;\WW^{-1,p}(\Omega)), \quad
F_w \in \L^s(0,T;\LL^r(\Omega)),
\quad
(v_0, w_0) \in \BB^{1-2/s}_{p,s}(\Omega) \times \LL^{r}(\Omega).
\end{eqnarray*}
From Assumption~\eqref{ass-max} made in section~\ref{sec-maxparaDisser}, the operator $\mathcal{A}_p$ owns the maximal parabolic regularity property over $\WW^{-1,p}(\Omega)$. Like in~\cite{Groeger1}, this can be achieved with the following resolvent estimate
\begin{eqnarray*}
\left\| (-\mathcal{A}_p +\lambda \I)^{-1}\right\|_{\mathscr{L}(\WW^{-1,p}(\Omega); \WW^{1,p}_{\#}(\Omega))}  \leq \frac{C}{|\lambda|}, & & \text{if } \mathrm{Re}\, \lambda \geq 0.
\end{eqnarray*}
From here we can derive the following result for system~\eqref{syslin}.

\begin{proposition} \label{prop-lin}
Let $T>0$. System~\eqref{syslin} admits a unique solution $(v,w) \in X_{s,p}^T \times Y_{s,r}^T$. It satisfies\footnote{Throughout the paper we will denote by $C$ a generic positive constant depending only on $\Omega$, $\nu$ and $\delta$, and which is in particular independent of $T$.}
\begin{eqnarray}
\|v\|_{X_{s,p}^T} + \|w\|_{Y_{s,r}^T} & \leq &
C\left(
\|v_0\|_{\BB^{1-2/s}_{p,s}(\Omega)} + \| w_0 \|_{\LL^{r}(\Omega)} +
\|F_v\|_{\L^s(0,T;\WW^{-1,p}(\Omega))} + \| F_w\|_{\L^s(0,T;\LL^r(\Omega))}
\right). \nonumber \\ \label{est-lin}
\end{eqnarray}
\end{proposition}

\begin{proof}
Note that the two evolution equations of~\eqref{syslin} are independent. First, in the same fashion as~\cite{Rehberg1}, we obtain existence and uniqueness of $v\in X_{s,p}^T$ solution of the first evolution equation, which satisfies
\begin{eqnarray*}
\|v\|_{X_{s,p}^T}  & \leq &
C\left(
\|v_0\|_{\BB^{1-2/s}_{p,s}(\Omega)}  + \|F_v\|_{\L^s(0,T;\WW^{-1,p}(\Omega))}\right).
\end{eqnarray*}
Second, the Duhamel's formula for the second evolution equation writes
\begin{eqnarray*}
w(t) & = & e^{-\delta t}w_0 + \int_0^t e^{-\delta(t-\tau)}F_w(\tau) \, \d \tau.
\end{eqnarray*}
This yields $\|w\|_{Y_{s,r}^T}   \leq C\left( \|w_0\|_{\LL^{r}(\Omega)}  + \|F_w\|_{\L^s(0,T;\LL^r(\Omega))}\right)$, and thus the announced estimate follows.
\end{proof}

\subsection{Wellposedness for the main system}

For proving existence of weak solutions for system~\eqref{sysmain}-\eqref{sysmain2}, we first need a set of technical lemmas.

\setcounter{lemma}{-1}

\begin{lemma} \label{lemma-basic}
Let $B$ be a Banach space. For $\varphi \in \W^{1,s}(0,T;B)$, one has
\begin{eqnarray*}
\|\varphi \|_{\L^s(0,T;B)} & \leq & T^{1/s}\|\varphi(0)\|_B
+ T \|\dot{\varphi} \|_{\L^s(0,T;B)}.
\end{eqnarray*}
\end{lemma}

\begin{proof}
Note that $\varphi$ is absolutely continuous. It follows that $\varphi(t) =  \varphi(0) + \int_0^t \dot{\varphi}$ and thus
\begin{eqnarray*}
\|\varphi(t)\|_B & \leq & \|\varphi(0)\|_B + \int_0^t \| \dot{\varphi}(\tau) \|_B \d \tau , \\
\|\varphi(t)\|_{\L^s(0,T;B)} & \leq & T^{1/s}\|\varphi(0)\|_B
+ \left(\int_0^T\left(\int_0^t \| \dot{\varphi}(\tau) \|_B \d \tau\right)^s \d t\right)^{1/s} \\
&  \leq & T^{1/s}\|\varphi(0)\|_B + \left( \int_0^T t^{s-1} \d t
\left(\int_0^T \| \dot{\varphi}(\tau) \|^s_B \d \tau\right)\right)^{1/s},
\end{eqnarray*}
leading to the announced estimate.
\end{proof}

The following lemma gives Lipschitz estimates for the PDE part.

\begin{lemma} \label{lemma-estlip1} \label{lemma-stab}
Let $\rho$ be a globally Lipschitz function on $\R$, 
with $C_{\rho}$ as Lipschitz constant.
For all $v_1$, $v_2 \in X_{s,p}^T$ and $w_1$, $w_2 \in Y^T_{s,r}$, we have
\begin{eqnarray}
\| \rho(v_1)-\rho(v_2) \|_{\L^s(0,T;\WW^{-1,p}(\Omega))} & \leq &
CC_{\rho}\max\left(T^{1/2},T^{1/(2s)}\right)\left(  \|v_1(0)-v_2(0)\|_{\BB^{1-2/s}_{p,s}(\Omega)} +  \|v_1-v_2\|_{X_{s,p}^T}
\right), \qquad  \label{est-lip2}\label{est-lip1} \\
\| \rho(w_1)-\rho(w_2) \|_{\L^s(0,T;\WW^{-1,p}(\Omega))} & \leq & CC_{\rho} T^{1/s} \| w_1-w_2\|_{C([0,T];\LL^r(\Omega))} .
\label{est-lip22} \label{est-lip11}
\end{eqnarray}
\end{lemma}

\begin{proof}
From~\eqref{ass-r2} we have the continuous embedding $\LL^p(\Omega) \hookrightarrow \WW^{-1,p}(\Omega)$, and so
$\|\rho(v_1) - \rho(v_2)\|_{\WW^{-1,p}(\Omega)}  \leq C\|\rho(v_1)-\rho(v_2)\|_{\LL^{p}(\Omega)}$. The Lipschitz continuity of $\rho$ yields
\begin{eqnarray*}
\|\rho(v_1)-\rho(v_2) \|_{\WW^{-1,p}(\Omega)} & \leq & CC_{\rho}\|v_1-v_2\|_{\LL^{p}(\Omega)}, \\
\| \rho(v_1)-\rho(v_2) \|_{\L^s(0,T;\WW^{-1,p}(\Omega))} & \leq &
	C_{\rho} \|v_1-v_2 \|_{\L^s(0,T;\LL^{p}(\Omega))}.
\end{eqnarray*}
Now denote $v=v_1-v_2$. Since by interpolation we have $\LL^p(\Omega) = [\W^{-1,p}(\Omega);\W^{1,p}(\Omega)]_{1/2}$, we deduce
\begin{eqnarray*}
	\|v \|_{\L^s(0,T;\LL^{p}(\Omega))}
	& \leq & C \|v \|^{1/2}_{\L^s(0,T;\WW^{-1,p}(\Omega))}
	\|v \|^{1/2}_{\L^s(0,T;\WW^{1,p}(\Omega))} \\
	& \leq & C \left(T^{1/s}\|v(0)\|_{\WW^{-1,p}(\Omega)} + T\|v \|_{\W^{1, s}(0,T;\WW^{-1,p}(\Omega))}\right)^{1/2}
	\|v \|^{1/2}_{\L^s(0,T;\WW^{1,p}(\Omega))}\\
& \leq & C T^{1/(2s)}\left(\|v(0)\|_{\WW^{-1,p}(\Omega)} + T^{1-1/s}\|v \|_{\W^{1, s}(0,T;\WW^{-1,p}(\Omega))}\right)^{1/2}
	\|v \|^{1/2}_{\L^s(0,T;\WW^{1,p}(\Omega))},
\end{eqnarray*}
where we have used Lemma~\ref{lemma-basic}. Since $(a+b)^{1/2} \leq  a^{1/2}+b^{1/2}$ for all $a\geq 0$ and $b \geq 0$, we deduce
\begin{eqnarray*}
\|v \|_{\L^s(0,T;\LL^{p}(\Omega))} & \leq &
C T^{1/(2s)}\left( \|v(0)\|^{1/2}_{\WW^{-1,p}(\Omega)} + T^{1/2-1/(2s)}\|v \|^{1/2}_{\W^{1, s}(0,T;\WW^{-1,p}(\Omega))}\right)
\|v \|^{1/2}_{\L^s(0,T;\WW^{1,p}(\Omega))},
\end{eqnarray*}
and thus~\eqref{est-lip2} follows with the Young's inequality and by using $\BB^{1-2/s}_{p,s}(\Omega) \hookrightarrow \WW^{-1,p}(\Omega)$. From the continuous embedding $\LL^r(\Omega) \hookrightarrow \WW^{-1,p}(\Omega)$, obtaining~\eqref{est-lip22} is straightforward.
\end{proof}

Let us now give estimates for the non-linearity of the ODE part.
\begin{lemma} \label{lemma-estlip2}
Let $\rho$ be a Lipschitz function on $\R$. 
Denote by $C_{\rho}$ its Lipschitz constant.
For all $v_1, v_2 \in X_{s,p}^T$ and $w_1, w_2 \in C([0,T];\LL^r(\Omega))$, we have
\begin{eqnarray}
\| \rho(v_1)-\rho(v_2) \|_{\L^s(0,T;\LL^r(\Omega))} & \leq & CC_{\rho}\max\left(T^{1/2},T^{1/(2s)}\right)\left(  \|v_1(0)-v_2(0)\|_{\BB^{1-2/s}_{p,s}(\Omega)} + \|v_1-v_2\|_{X_{s,p}^T}\right), \qquad \label{est-lip5}\label{est-lip3}\\
\| \rho(w_1)-\rho(w_2) \|_{\L^s(0,T;\LL^r(\Omega))} & \leq & CC_{\rho}T^{1/s} \|w_1-w_2\|_{C([0,T];\LL^r(\Omega))}.
\label{est-lip6}\label{est-lip4}
\end{eqnarray}
\end{lemma}

\begin{proof}
The proof is very similar to the one of Lemma~\ref{lemma-estlip1}, and so its verification left to the reader.
\end{proof}

From Lemma~\ref{lemma-estlip1} and \textcolor{black}{Lemma}~\ref{lemma-estlip2}, we deduce Lipschitz estimates for multi-dimensional nonlinear mappings.

\begin{lemma} \label{lemma-mu}
Assume that $\mu$ is a Lipschitz function from $\R^2$ to $\R$. 
For all $(v_1,w_1), (v_2,w_2) \in X^T_{s,p} \times Y^T_{s,r}$, we have
\begin{eqnarray*}
& & \| \mu(v_1,w_1) - \mu(v_2,w_2) \|_{\L^s(0,T;\WW^{-1,p}(\Omega))}
+ \| \mu(v_1,w_1) - \mu(v_2,w_2) \|_{\L^s(0,T;\LL^{r}(\Omega))} \\
& & \leq  CC_{\mu} \left(
\max\left(T^{1/2},T^{1/(2s)}\right)\left(  \|v_1(0)-v_2(0)\|_{\BB^{1-2/s}_{p,s}(\Omega)} + \|v_1-v_2\|_{X_{s,p}^T}\right)  + T^{1/s} \|w_1-w_2\|_{C([0,T];\LL^r(\Omega))} \right), \label{est-lip-mu2}\label{est-lip-mu1}
\end{eqnarray*}
where $C_{\mu}$ denotes the Lipschitz constant of~$\mu$.
\end{lemma}

\begin{proof}
The continuous embeddings $\LL^p(\Omega) \hookrightarrow \LL^r(\Omega) \hookrightarrow \WW^{-1,p}(\Omega)$ yield
\begin{eqnarray*}
\|\mu(v_1,w_1)- \mu(v_2,w_2)\|_{\WW^{-1,p}(\Omega)} & \leq &  C\| \mu(v_1,w_1)- \mu(v_2,w_2) \|_{\LL^r(\Omega)} \\
& \leq & CC_{\mu} \left( \| v_1-v_2 \|_{\LL^r(\Omega)} + \| w_1-w_2 \|_{\LL^r(\Omega)} \right) \\
& \leq & CC_{\mu} \left( \| v_1-w_2 \|_{\LL^p(\Omega)} + \| w_1-w_2 \|_{\LL^r(\Omega)} \right),
\end{eqnarray*}
since $r \leq p$ from~\eqref{ass-r}. From here, one can combine the steps of the respective proofs of Lemma~\ref{lemma-estlip1} and Lemma~\ref{lemma-estlip2}, and then the results follow.
\end{proof}

From Lemma~\ref{lemma-mu} and Proposition~\ref{prop-lin}, let us derive a general existence result for the monodomain model with semilinear operators. This result will also be used in section~\ref{sec-lin} for obtaining local wellposedness of non-autonomous linear systems.

\begin{proposition} \label{prop-locexist}
Consider two Lipschitz mappings $\mu_v$ and $\mu_w$ from $\R^2$ to $\R$, such that $\mu_v(0,0) =\mu_w(0,0)= 0$. Then for $(v_0,w_0) \in \BB^{1-2/s}_{p,s}(\Omega) \times \LL^r(\Omega)$ and $(f_v,f_w) \in \L^s_{\mathrm{loc}}(0,\infty;\WW^{-1,p}(\Omega)) \times \L^s_{\mathrm{loc}}(0,\infty;\LL^r(\Omega))$, there exists a maximal time $T_0>0$ such that the following system
\begin{eqnarray}
\left\{ \begin{array} {ll}
\displaystyle \frac{\p v}{\p t} - \nu\Delta v + \mu_v(v,w) = f_v &
\text{in } \Omega\times (0,T),\\[10pt]
\displaystyle \frac{\p w}{\p t} + \delta w + \mu_w(v,w) = f_w &
\text{in } \Omega\times (0,T),\\[10pt]
\displaystyle \frac{\p v}{\p n} = 0 & \text{on } \p \Omega \times (0,T),\\[10pt]
(v,w)(\, \cdot \, ,0) = (v_0,w_0) & \text{in } \Omega.
\end{array} \right. \label{sysgeneral}
\end{eqnarray}
admits a unique weak solution in $X_{s,p}^T \times Y_{s,r}^T$ for all $0<T<T_0$. Moreover, the following estimate holds:
\begin{eqnarray}
\| (v,w) \|_{X^T_{s,p}\times Y^T_{s,r}} & \leq & C
\left( \|(v_0,w_0) \|_{\BB^{1-2/s}_{p,s}(\Omega) \times \L^r(\Omega)} +
 \|(f_v,f_w) \|_{\L^s(0,T;\W^{-1,p}(\Omega)) \times \L^s(0,T;\L^r(\Omega))} \right).
\label{estexistlocal-mu}
\end{eqnarray}
The lifespan $T_0$ depends on $(v_0,w_0)$, $(f_v,f_w)$, and decreasingly on the Lipschitz constants of $\mu_v$ and $\mu_w$.
\end{proposition}

\begin{proof}
Define the mapping
\begin{eqnarray*}
\mathcal{K}: X^T_{s,p} \times Y^T_{s,r} \ni (v_1,w_1)  \mapsto (v_2,w_2) \in X^T_{s,p} \times Y^T_{s,r},
\end{eqnarray*}
where $(v_2,w_2)$ is the solution of system~\eqref{syslin}, with $F = f_v - \mu_v(v_1,w_1)$ and $G = f_w - \mu_w(v_1,w_1)$ as right-and-sides. Define the following closed set
\begin{eqnarray*}
\mathcal{B}_T & = & \left\{ (v,w) \in X^T_{s,p} \times Y^T_{s,r} \mid \, \|(v,w) \|_{X^T_{s,p} \times Y^T_{s,r}} \leq 2CR
\right\},
\end{eqnarray*}
where the constant $C$ is the one which appears in estimate~\eqref{est-lin}, and
\begin{eqnarray*}
R & = & \|v_0\|_{\BB^{1-2/s}_{p,s}(\Omega)} + \| w_0 \|_{\LL^{r}(\Omega)} +
\|f_v\|_{\L^s(0,T;\WW^{-1,p}(\Omega))} + \| f_w\|_{\L^s(0,T;\LL^r(\Omega))}.
\end{eqnarray*}
Using~\eqref{est-lin} of Proposition~\ref{prop-lin}, combined with Lemma~\ref{lemma-mu}, by choosing $T$ small enough\textcolor{black}{, namely
\begin{equation*}
C_{\mu}\max(T^{1/s}, T^{1/(2s)}) \leq CR,
\end{equation*}}one can guarantee that $\mathcal{B}_T$ invariant under $\mathcal{K}$, and that $\mathcal{K}$ is a contraction on $\mathcal{B}_T$. Thus $\mathcal{K}$ admits a unique fixed point, namely a unique solution of system~\eqref{sysgeneral}. By a classical continuation argument, there exists a maximal $T_0$.
\end{proof}

In order to apply Proposition~\ref{prop-locexist} to system~\eqref{sysmain}-\eqref{sysmain2} \textcolor{black}{with the neural networks $(\phi_v,\phi_w)$ in the role of the mappings~$\mu_v$ and~$\mu_w$, one needs to verify that these neural networks given by~\eqref{sysmain2} fulfill the assumptions made for~$\mu_v$ and~$\mu_w$, that is they are Lipschitz}. A direct consequence of Lemmas~\ref{lemma-estlip1} and~\ref{lemma-estlip2} are the following stability properties for the neural networks:

\begin{lemma} \label{lemma-estlip3} \label{lemma-nem}
Assume that the activation function $\rho$ is globally Lipschitz. The Nemytskii operators $\phi_v$ and $\phi_w$ defined in~\eqref{sysmain2} are Lipschitz continuous from $\L^s(0,T;\LL^r(\Omega))\times \L^s(0,T;\LL^r(\Omega))$ to $\L^s(0,T;\LL^r(\Omega))$, with Lipschitz constant less than or equal to $C_{\rho}^{\mathcal{L}_{\mathcal{W}}}\|\mathbf{W}\|_{\mathcal{W}}^L$, \textcolor{black}{where $C_{\rho}$ denotes the Lipschitz constant of the activation function $\rho$, $L$ is the number of layers of $(\phi_v,\phi_w)$, and where we recall $\mathcal{L}_{\mathcal{W}} = \displaystyle \sum_{\ell=1}^{L-1} n_{\ell}$}. Furthermore, we have
\begin{eqnarray}
& & \|\phi_v(v_1,w_1) -\phi_v(v_2,w_2) \|_{\L^s(0,T;\WW^{-1,p}(\Omega))}
+ \|\phi_w(v_1,w_1) -\phi_w(v_2,w_2) \|_{\L^s(0,T;\LL^r(\Omega))} \nonumber \\
& & \leq
CC_{\rho}^{\mathcal{L}_{\mathcal{W}}} \| \mathbf{W}\|_{\mathcal{W}}^L \left( \max\left(T^{1/2},T^{1/(2s)}\right)
\left( \|v_1(0)-v_2(0)\|_{\BB^{1-2/s}_{p,s}(\Omega)} + \|v_1-v_2\|_{X_{s,p}^T} \right) \right. \nonumber \\
& & \left. + T^{1/s} \|w_1-w_2\|_{C([0,T];\LL^r(\Omega))}  \right) \label{super-estlip}
\end{eqnarray}
\textcolor{black}{for all $(v_1,w_1)$, $(v_2,w_2) \in X^T_{s,p} \times Y^T_{s,r}$.}
\end{lemma}

\begin{proof}
The Lipschitz constants of $\phi_v$ and $\phi_w$ are controlled by the one of $\Phi = (\phi_v,\phi_w)$, which satisfies estimate~\eqref{resLip}. Then the results follow from Lemma~\ref{lemma-mu}.
\end{proof}

Then  Lemma~\ref{lemma-nem} and Proposition~\ref{prop-locexist} imply  local wellposedness of system~\eqref{sysmain}-\eqref{sysmain2}:

\begin{theorem} \label{th-exist0}
Assume that~\eqref{ass-ps}-\eqref{ass-r} hold, and that $\rho$ is globally Lipschitz. Then for $(v_0,w_0) \in \BB^{1-2/s}_{p,s}(\Omega) \times \LL^r(\Omega)$ and $(f_v,f_w) \in \L^s_{\mathrm{loc}}(0,\infty;\WW^{-1,p}(\Omega)) \times \L^s_{\mathrm{loc}}(0,\infty;\LL^r(\Omega))$, there exists a maximal time $T_0$ such that system~\eqref{sysmain}-\eqref{sysmain2} admits a unique weak solution in $X_{s,p}^T \times Y_{s,r}^T$ for all $T<T_0$. Moreover, the following estimate holds:
\begin{eqnarray}
\| (v,w) \|_{X^T_{s,p}\times Y^T_{s,r}} & \leq & CC_{\rho}^{\mathcal{L}_{\mathcal{W}}}\|\mathbf{W}\|_{\mathcal{W}}^L
\left(1+ \|(v_0,w_0) \|_{\BB^{1-2/s}_{p,s}(\Omega) \times \L^r(\Omega)} +
 \|(f_v,f_w) \|_{\L^s(0,T;\W^{-1,p}(\Omega)) \times \L^s(0,T;\L^r(\Omega))} \right). \nonumber \\
\label{estexistlocal}
\end{eqnarray}
The lifespan $T_0$ depends on $(v_0,w_0)$, $(f_v,f_w)$ and \textcolor{black}{decreasingly} on $C_{\rho}^{\mathcal{L}_{\mathcal{W}}}\|\mathbf{W}\|_{\mathcal{W}}^L$. 
\end{theorem}

\begin{proof}
This a direct consequence of Proposition~\ref{prop-locexist} with $\mu_v = \phi_v - \phi_v(0)$, $\mu_w = \phi_w - \phi_w(0)$, and $(f_v,f_w)$ replaced by $(f_v-\phi_v(0),f_w-\phi_w(0))$. Note that $\phi_v(0)$ and $\phi_w(0)$ are controlled as follows:
\begin{linenomath}\begin{equation*}
\|(\phi_v(0), \phi_w(0)) \|_{\WW^{-1,p}(\Omega) \times \LL^r(\Omega)} \leq
C C_{\rho}^{\mathcal{L}_{\mathcal{W}}} \|\mathbf{W}\|_{\mathcal{W}}^L.
\end{equation*}\end{linenomath}
Thus the result follows from the Lipschitz properties of $(\phi_v,\phi_w)$ given in Lemma~\ref{lemma-nem}.
\end{proof}

With additional assumptions, one can make the solutions of system~\eqref{sysmain}-\eqref{sysmain2} global.

\begin{theorem} \label{th-global} \label{th-exist}
Assume that the assumptions of Theorem~\ref{th-exist0} hold. In addition to~\eqref{ass-ps}-\eqref{ass-r}, assume that $s\leq 2$ and $r \leq 2$, and that $(f_v,f_w) \in \L_{\mathrm{loc}}^2(0,\infty;\LL^2(\Omega)) \times \L_{\mathrm{loc}}^2(0,\infty;\LL^2(\Omega))$. Then system~\eqref{sysmain}-\eqref{sysmain2} admits a unique global solution in $X_{2,p}^\infty$.
\end{theorem}

\begin{proof}
We first prove an energy estimate for system~\eqref{sysmain}-\eqref{sysmain2} on $(0,T_0)$, where $T_0$ is the maximal time of existence, assumed to be finite, and next proceed by contradiction.\\
{\it Step 1. } By multiplying the first and second equations of~\eqref{sysmain} by $v$ and $w$, respectively, integrating, and summing the two resulting identities, we obtain the following equality:
\begin{linenomath}\begin{equation}
\frac{1}{2}\frac{\d}{\d t}\left(\|v\|^2_{\LL^2(\Omega)} + \|w\|^2_{\LL^2(\Omega)}\right)
+ \nu \|\nabla v\|^2_{[\LL^2(\Omega)]^d} + \delta \|w\|_{\LL^2(\Omega)}^2  =  \left\langle (f_v,f_w) - \Phi(v,w); (v,w) \right\rangle_{\LL^2(\Omega) \times \LL^2(\Omega)}.
\label{L2-est}
\end{equation}\end{linenomath}
Besides, it is easy to prove the following estimate:
\begin{eqnarray*}
\| \Phi(v,w)\|_{\LL^2(\Omega)\times \LL^2(\Omega)} & \leq &
C C_{\rho}^{\mathcal{L}_{\mathcal{W}}} \|\mathbf{W}\|_{\mathcal{W}}^L
\left(1+\| (v,w) \|_{\LL^2(\Omega) \times \LL^2(\Omega)}\right).
\end{eqnarray*}
Using the Cauchy-Schwarz inequality combined with the previous estimate, equality~\eqref{L2-est} yields
\begin{linenomath}\begin{equation*}
\frac{\d}{\d t}\left(\|v\|^2_{\LL^2(\Omega)} + \|w\|^2_{\LL^2(\Omega)}\right)
+ \nu \|\nabla v\|^2_{[\LL^2(\Omega)]^d} + \delta \|w\|_{\LL^2(\Omega)}^2  \leq  C\left(1+\|(f_v,f_w)\|_{\LL^2(\Omega)\times \LL^2(\Omega)}^2 + \|v\|^2_{\LL^2(\Omega)} + \|w\|^2_{\LL^2(\Omega)} \right),
\end{equation*}\end{linenomath}
in $(0,T_0)$. Assuming that $T_0 < \infty$. The Gr\"{o}nwall's lemma then yields
\begin{eqnarray}
& & \sup_{(0,T_0)}\left\{(\|v\|^2_{\LL^2(\Omega)} + \|w\|^2_{\LL^2(\Omega)} \right\}
+ \nu \int_0^{T_0} \|\nabla v\|^2_{[\LL^2(\Omega)]^d} \d t
+ \delta  \int_0^{T_0} \|w\|^2_{\LL^2(\Omega)} \d t \nonumber \\
& & \leq
\left(1+\|(f_v,f_w)\|_{\L^2(0,T_0;\LL^2(\Omega)\times \LL^2(\Omega))}^2 + \|v_0\|^2_{\LL^2(\Omega)} + \|w_0\|^2_{\LL^2(\Omega)}  \right)\exp(CT_0),
\label{estenergy}
\end{eqnarray}
where the constant $C>0$ is generic, but depending on~$\mathbf{W}$ and $C_{\rho}$.

{\it Step 2. } Inequality~\eqref{estenergy} shows that $t\mapsto \|v(t)\|^2_{\LL^2(\Omega)} + \|w(t)\|^2_{\LL^2(\Omega)}$ is bounded on $(0,T_0)$. Besides, for $s\leq 2$ and $r\leq 2$ we have $\LL^2(\Omega) \hookrightarrow \BB^{1-2/s}_{p,s}(\Omega)$ and $\LL^2(\Omega) \hookrightarrow \L^r(\Omega)$, respectively. Then from $v(T_0), w(T_0) \in \LL^2(\Omega)$ we can extend the solution $t\mapsto (v(t),w(t))$ on $(0,T_0 + \epsilon)$, for some $\varepsilon>0$, in virtue of Theorem~\ref{th-exist0}. This contradicts the definition of $T_0$ as an upper bound, and thus $T_0 = \infty$.
\end{proof}

\subsection{Global existence results for other functional frameworks}

Let us \textcolor{black}{investigate the applicability} of the results in the literature dealing with the issue of global existence in this context of semilinear equations, either by modifying regularity on the data, or by considering more general nonlinearities.

\subsubsection{For initial conditions in critical spaces}

Following the results of~\cite{Wilke2018} dealing with solutions of parabolic equations in weighted Sobolev spaces, with initial conditions considered in the so-called {\it critical spaces}, one can weaken the regularity assumption required for the initial conditions $(v_0,w_0)$. In this fashion,~\cite[Theorem~3.2, section~3.2]{Hieber2018} states that $v_0$, for the FitzHugh-Nagumo model, can be actually considered in $\BB^{2\mu - 1 - 2/s}_{p,s}$ for $\mu \geq \mu_c := 1/s + d/(2p)$. In our case, the nonlinearity is Lipschitz (and not negative cubic like in~\cite{Hieber2018}), and thus~\cite[Theorem~2.1]{Wilke2018} yields the following result:

\begin{theorem}
Assume that $1<s<p<\infty$. For $(v_0,w_0) \in \BB^{-1+\varepsilon}_{p,s}(\Omega) \times \LL^r(\Omega)$, with $0<\varepsilon \leq 2-2/s$, and $(f_v,f_w) \in \L_{\mathrm{loc}}^s(0,\infty;\WW^{-1,p}(\Omega)) \times \L_{\mathrm{loc}}^s(0,\infty;\LL^r(\Omega))$. Then there exists $T_0 >0$, and a unique solution $(v,w)$ to system~\eqref{sysmain}-\eqref{sysmain2} such that
\begin{eqnarray*}
v\in X_{s,p,\varepsilon}^{T_0}   \stackrel{\mathrm{def}}{\Longleftrightarrow}
\left\{\begin{array}{l}
 t^{1-1/s-\varepsilon/2}v \in \L^s(0,T_0; \WW^{1,p}(\Omega)) \\ 
 t^{1-1/s-\varepsilon/2}\dot{v} \in \L^s(0,T_0; \WW^{-1,p}(\Omega))
\end{array} \right. ,
 & &
w \in Y_{s,r}^{T_0},
\end{eqnarray*}
and the following alternative holds:
\begin{eqnarray*}
& (i) & \text{either } T_0 = \infty,\\
& (ii) & \text{or } \lim_{t\mapsto T_0} \left( \| v(t)\|_{B^{-1+\varepsilon}_{p,s}(\Omega)} + \| w(t) \|_{\L^r(\Omega)} \right) = \infty
\end{eqnarray*}
Furthermore, if $r\leq 2$ and $v_0, w_0 \in \LL^2(\Omega)$, then $T_0 = \infty$.
\end{theorem}

\begin{proof}
The proof is again a fixed-point argument. First observe that the trace space of $ X_{s,p,\varepsilon}^{T}$ is $\BB^{-1+\varepsilon}_{p,s}(\Omega)$. In view of what precedes, let us just sketch the proof. First,~\cite[Theorem~2.1]{Wilke2018} applies to~\eqref{syslin} with $v_0 \in \BB^{2\mu -1 - 2/s}_{p,s}(\Omega)$ for every $\mu > 1/s$ arbitrarily small. Denote $\varepsilon = 2\mu - 2/s$. Then~\eqref{est-lin} holds with $\|v_0\|_{\BB^{-1+\varepsilon}_{p,s}(\Omega)}$ instead of $\|v_0\|_{\BB^{1-2/s}_{p,s}(\Omega)}$. Next, in view of the estimate of Lemma~\ref{lemma-mu}, and since $\BB^{-1+\varepsilon}_{p,s}(\Omega) \hookrightarrow \BB^{-1}_{p,p}(\Omega) \equiv \WW^{-1,p}(\Omega)$ for $p\geq s$, the same fixed-point strategy can be repeated as for proving Proposition~\ref{prop-locexist} while assuming $v_0 \in \BB^{-1+\varepsilon}_{p,s}(\Omega)$ only, and replacing $X_{s,p}^T$ by $X_{s,p,\varepsilon}^T$. Alternative~$(i)$-$(ii)$ is obtained with a classical continuation argument. Furthermore, by noticing that $\LL^2(\Omega) \hookrightarrow \BB^{-1+\varepsilon}_{s,p}(\Omega)$, the proof of Theorem~\ref{th-global} can also be repeated, and thus we obtained the last assertion.
\end{proof}

\subsubsection{In the case of low regularity data}

In~\cite{Amann2004} a weaker functional setting is studied for semilinear parabolic equations of type
\begin{eqnarray*}
\dot{v} - \nu \Delta v = \mu(v),
\end{eqnarray*}
where the operator\footnote{Here the notation $C^{1-}_b$ means {\it Lipschitz and bounded}.} $\mu \in C^{1-}_b\left(\L^s_{\text{loc}}(\R^+;\WW^{1,p}(\Omega)); \mathcal{M}_{\mathrm{loc}}(\R^+; \WW^{\sigma-2,p}(\Omega)) \right)$, for any $1<\sigma < 2$, has the so-called {\it Volterra} property (see~\cite[Theorem~3.2]{Amann2004}). The solution lies only in $\L^s(0,T;\W^{1,p}(\Omega))$, assuming that $1 \leq s < 2/(3-\sigma)$. The dependence of the solution is Lipschitz with respect to the operator $\mu$. Next this framework was considered in~\cite{Amann2006} for optimal control problems involving this class of semilinear equations as state constraints. It is particularly well-chosen when the initial data lie in measure spaces. In the case of the present article, the specific form of the nonlinearities given by $\phi_v$ offers more regularity than $\mathcal{M}_{\mathrm{loc}}(\R^+; \WW^{\sigma-2,p}(\Omega))$, and so such a functional setting would be relevant only when considering initial data in measure spaces, up to replacing $\mu$ by $\mu + v_0 \otimes \delta_0$.

\subsubsection{Semilinear equations with more general nonlinearity.} \label{secthAmann}

Let us mention that in~\cite[section~3]{Amann2004} a general result is given for semilinear equations with other assumptions on the nonlinearity. Thus~\cite[Theorem~3.3]{Amann2004} states in a particular case that if $1<\sigma < \min(2, 1+2/s)$ with $\sigma, \sigma-2/s \in [-2;2] \setminus\{1+1/p; -2+1/p\}$, and if
\begin{linenomath}\begin{equation*}
\mu \in C^{1-}_b\left(\L^s_{\text{loc}}(\R^+;\WW^{1,p}(\Omega)); \L^s_{\text{loc}}(\R^+; \WW^{\sigma-2,p}(\Omega))\right)
\end{equation*}\end{linenomath}
has the {\it Volterra} property, then, for all $v_0 \in \WW^{\sigma-2/s,p}(\Omega)$, the following system
\begin{eqnarray*}
\dot{v} -\nu \Delta v = \mu(v), \ t\in (0,\infty), & & v(0) = v_0,
\end{eqnarray*}
admits a unique maximal solution $v\in \L^s_{\text{loc}}(\R^+;\WW^{1,p}(\Omega))$. Moreover,
the mapping
\begin{eqnarray*}
\WW^{\sigma-2/s,p} \times C^{1-}_b\left(\L^s_{\mathrm{loc}}(\R^+;\WW^{1,p}(\Omega)); \L_{\mathrm{loc}}^s(\R^+; \WW^{\sigma-2,p}(\Omega))\right)
\ni (v_0, \mu) & \mapsto & v \in \L^s_{\mathrm{loc}}(\R^+;\WW^{1,p}(\Omega))
\end{eqnarray*}
is Lipschitz continuous.


Note that such a functional framework does not straightforwardly apply to our coupling between the semilinear equation and the ordinary differential equation. However, this framework could be adopted for treating optimal control problems with other models of parabolic equations as state constraints.

\section{Formulation of an optimal control problem and existence of minimizers} \label{sec-opt}

Recall that we assumed $1<s$, $1<p<d$ and $(dp)/(d+p) <r \leq p$. \textcolor{black}{We are interested in the following abstract optimal control problem:
\begin{linenomath}\begin{equation} \label{mainpb} \tag{$\mathcal{P}$}
\left\{ \begin{array} {l}
\displaystyle \min_{\mathbf{W} \in \mathcal{W}} J(\mathrm{z},\mathbf{W})
, \\[10pt]
\text{subject to $(\mathrm{z}=(v,w),\bW)$ satisfying~\eqref{sysmain}-\eqref{sysmain2}, and } \| \mathbf{W}\|^2_{\mathcal{W}} \leq C. 
\end{array} \right.
\end{equation}\end{linenomath}
We assume that the functional $J:X^T_{s,p} \times Y_{s,r}^T \rightarrow \R$ is only continuously Fr\'echet-differentiable. For example, still by denoting $\mathrm{z} = (v,w)$, let us consider smooth convex objective functions in the following form:}
\begin{eqnarray}
J(\mathrm{z},\mathbf{W}) & = & \displaystyle
R(\mathrm{z}) + R_T(\mathrm{z}), \quad \text{where} \label{choice-func} \\
R(\mathrm{z}) & = & \frac{1}{s} \int_0^T \|v - v_{\mathrm{data}} \|_{\LL^r(\Omega)}^s\d t
+ \frac{1}{s} \int_0^T\| w - w_{\mathrm{data}}\|^s_{\LL^r(\Omega)}\d t, \nonumber \\
R_T(\mathrm{z}) & = &
\frac{1}{2}\|v(T) - v^{(T)}_{\mathrm{data}}\|_{\LL^2(\Omega)}^2
+ \frac{1}{r}\|w(T) - w^{(T)}_{\mathrm{data}} \|^r_{\LL^r(\Omega)}. \nonumber
\end{eqnarray}
\textcolor{black}{Proving that $R$, $R_T$, and consequently $J$ are continuously Fr\'echet-differentiable can be achieved with standard techniques from optimal control theory.} 

Note that \textcolor{black}{since the set of control parameters $\mathbf{W}$ is compact, Theorem~\ref{th-exist0} guarantees} the existence of solutions $\mathrm{z}$ of system~\eqref{sysmain}-\eqref{sysmain2} for which the lifespan $T_0$ is independent of the control parameters $\mathbf{W}$. For this reason we assume that $T \in (0,T_0)$.

\begin{remark}
One could also consider problem~\eqref{mainpb} without a norm constraint on $\bW$. In that case, in order to have a uniform lifespan $T_0 >0$ for the state equation, the restricted framework of Theorem~\ref{th-global}, namely $s\leq 2$ and $r\leq 2$,  could be adopted to have $T_0 = \infty$.
\end{remark}

\begin{proposition} \label{prop-exist-min}
Problem~\eqref{mainpb} admits a global minimizer.
\end{proposition}

\begin{proof}
The existence of feasible solutions is due to Theorem~\ref{th-exist}. Since $R$ is bounded from below, problem~\eqref{mainpb} admits a minimizing sequence $(\mathrm{z}_n, \mathbf{W}_n)_n = (v_n,w_n,\mathbf{W}_n)_n$ of feasible solutions. Since $\|\mathbf{W}_n\|_{\mathcal{W}}$ is bounded uniformly with respect to $n$, from~\eqref{estexistlocal} the sequence $\mathrm{z}_n$ is uniformly bounded as well, and admits a uniform positive time of existence. So, up to extraction, there exists $\hat{\mathrm{z}} \in X^T_{s,p} \times Y_{s,r}^T$ and $\hat{\mathbf{W}} \in \mathcal{W}$ such that
\begin{eqnarray*}
\mathrm{z}_n \rightharpoonup \hat{\mathrm{z}}=(\hat{v},\hat{w}) \text{ in } X^T_{s,p} \times Y^T_{s,r},
& &  \mathbf{W}_n \rightarrow \hat{\mathbf{W}} \text{ in } \mathcal{W}.
\end{eqnarray*}
The convergence of $(\mathbf{W}_n)_n$ is strong because $\mathcal{W}$ is finite-dimensional. Further, the bounds provided by estimate~\eqref{estexistlocal} for the norm of $\mathrm{z}_n$ in $X^T_{s,p} \times Y^T_{s,r}$, yield weak convergence (up to extraction) of
\begin{eqnarray}
& & \frac{\p v_n}{\p t} \text{ in } \L^s(0,T; \WW^{-1,p}(\Omega)), \quad
\nabla v_n \text{ in } \L^s(0,T; \LL^{p}(\Omega)), \quad \frac{\p w_n}{\p t} \text{ in } \L^s(0,T;\LL^r(\Omega)), \label{weakcvlin} \\
& &  (v_n(T),w_n(T)) \text{ in } \BB^{1-2/s}_{p,s}(\Omega) \times \LL^r(\Omega). \nonumber
\end{eqnarray}
Since, by the  Rellich-Kondrachov theorem,   the embedding $\WW^{1,p}(\Omega) \hookrightarrow \LL^{p^\ast}(\Omega)$ is compact with $p^\ast := pd/(d-p) > r $ (see~\eqref{ass-r2}), the sequence $v_n$ converges strongly in $\L^s(0,T; \LL^r(\Omega))$, towards $\hat{v}$. On the other hand,
the functional $J$ is weakly lower semi-continuous. 
We deduce
\begin{eqnarray}
 \lim_{n\rightarrow \infty} J(\mathrm{z}_n,\mathbf{W}_n) \geq
\liminf_{n\rightarrow \infty} J(\mathrm{z}_n,\mathbf{W}_n) \geq
J(\hat{\mathrm{z}},\hat{\mathbf{W}}). \label{great-ineqs}
\end{eqnarray}
This shows that $(\hat{\mathrm{z}},\hat{\mathbf{W}})$ minimizes the functional $J$. To verify that it is a solution of problem~\eqref{mainpb}, we have to show that it satisfies~\eqref{sysmain}-\eqref{sysmain2}. In view of the form of $J$, and the strong convergence of $v_n \rightarrow v$ and $\mathbf{W}_n \rightarrow \mathbf{W}$ in $\L^s(0,T;\LL^r(\Omega))$ and $\mathcal{W}$, respectively, from~\eqref{great-ineqs} it follows that\textcolor{black}{
\begin{eqnarray*}
\|w_n - w_{\mathrm{data}} \|_{\L^s(0,T;\LL^r(\Omega))} & \rightarrow & \| w - w_{\mathrm{data}} \|_{\L^s(0,T;\LL^r(\Omega))}.
\end{eqnarray*}}
Since the Banach space $\L^s(0,T;\LL^r(\Omega))$ is uniformly convex (this is a consequence of the Hanner's inequalities), from~\cite[Proposition~3.32 page~78]{Brezis}, this is sufficient for deducing that \textcolor{black}{$w_n-w_{\mathrm{data}} \rightarrow w-w_{\mathrm{data}}$ strongly in $\L^s(0,T;\LL^r(\Omega))$, and thus $w_n \rightarrow w$ strongly}. Then, with Lemma~\ref{lemma-nem}, we can pass to the limit in the nonlinear terms of~\eqref{sysmain}. \textcolor{black}{Considering the weak formulation of~\eqref{sysmain} given in Definition~\ref{def-weak}, this implies weak convergence of the nonlinear terms. Further, from~\eqref{weakcvlin} we also have weak convergence} for the linear terms of~\eqref{sysmain}, and therefore $(\hat{v},\hat{w})$ is a weak solution corresponding to~$\hat{\mathbf{W}}$. Thus the proof is complete.
\end{proof}


While we can prove existence of minimizers for problem~\eqref{mainpb}, the derivation of necessary optimal conditions requires a regularization of the activation function.

\section{Regularization for derivation of optimality conditions} \label{sec-reg}

\subsection{Regularization} \label{sec-assump}
Assuming that the activation functions $\rho$ are only Lipschitz is not sufficient for deriving optimality conditions with standard tools. For this purpose, we consider an approximation $\rho_{\varepsilon}$ of the activation function $\rho$, \textcolor{black}{that we assume to satisfy} the following set of assumptions:
\begin{description}
\item[$(\mathbf{A1})$] For all $\varepsilon>0$, the function $\rho_{\varepsilon}$ is of class $C^1$ on $\R$.

\item[$(\mathbf{A2})$] For all $\varepsilon>0$, we have $\|\rho'_{\varepsilon}\|_{\L^\infty(\R)} \leq \|\rho'\|_{\L^\infty(\R)}$.

\item[$(\mathbf{A2})'$] For all $\varepsilon>0$, we have $0 \leq \rho'_{\mathrm{min}} \leq \rho'_{\varepsilon} \leq \rho'_{\mathrm{max}}$,
where $\rho'_{\mathrm{min}} := \mathrm{ess\,inf}\, \rho'$ and $\rho'_{\mathrm{max}} := \mathrm{ess\,sup}\, \rho'$.

\item[$(\mathbf{A3})$] The sequence $(\rho_{\varepsilon})_{\varepsilon}$ converges uniformly to $\rho$ in $C(\R;\R)$.
\end{description}

Assumption~$(\mathbf{A2})$ implies in particular that the Lipschitz constant of functions $\rho'_{\varepsilon}$ is bounded uniformly with respect to $\varepsilon$. Note that~$(\mathbf{A2})'$ implies~$(\mathbf{A2})$, and will be used only in section~\ref{sec-final}.

\paragraph{Example: The ReLU activation function.}The so-called ReLU ({\it Rectified Linear Units}) activation function $\rho : x \mapsto \max(0,x)$ can be approximated by
\begin{linenomath}\begin{equation}
\rho_{\varepsilon}(x)  =  \left\{
\begin{array} {ll}
0 & \text{if } x \leq -\varepsilon, \\[5pt]
\displaystyle\frac{1}{4\varepsilon}\left(x+\varepsilon\right)^2 &
\text{if } -\varepsilon \leq x \leq \varepsilon, \\[5pt]
x & \text{if } x \geq \varepsilon.
\end{array} \right. \label{smoothReLU}
\end{equation}\end{linenomath}
\textcolor{black}{One} can verify that $\rho_{\varepsilon}$ is of class $C^1$ on $\R$, and Lipschitz with constant $\leq 1$ independent of $\varepsilon$, and $(\rho_{\varepsilon})_{\varepsilon}$ converges uniformly towards $\rho$ in $C(\R; \R)$.\\

We are now interested in the following abstract optimal control problem:
\begin{linenomath}\begin{equation} \label{mainpbeps} \tag{$\mathcal{P}_\varepsilon$}
\left\{ \begin{array} {l}
\displaystyle \min_{\mathbf{W} \in \mathcal{W}} \big(J(\mathrm{z},\mathbf{W}) =
R(\mathrm{z})+ R_T(\mathrm{z}) \big) ,\\[10pt]
\text{subject to $(\mathrm{z},\bW)$ satisfying~\eqref{sysmain}-\eqref{sysmain2} with $\rho_\varepsilon$ as activation function, and } \| \mathbf{W}\|^2_{\mathcal{W}} \leq C,
\end{array} \right.
\end{equation}\end{linenomath}
\textcolor{black}{where $R$ and $R_T$ are given as in~\eqref{choice-func}.}

\subsection{Linear systems} \label{sec-lin}

\textcolor{black}{From now on, the dependence of $\phi_v$ and $\phi_w$ with respect to the $\mathbf{W}$ will be made explicit in the notation.}
Let be $\overline{\mathrm{z}} = (\overline{v},\overline{w}) \in X^T_{s,p} \times Y^T_{s,r}$ and $\overline{\mathbf{W}} \in \mathcal{W}$. Consider now neural networks $\Phi^{(\varepsilon)} = \Big(\phi^{(\varepsilon)}_v,\phi^{(\varepsilon)}_w\Big)$ with the regularized $(\rho_{\varepsilon})$ as activation functions. We introduce the following linearized system:
\begin{eqnarray}
\left\{ \begin{array} {ll}
\displaystyle \frac{\p v}{\p t} - \nu\Delta v +
\frac{\p \phi^{(\varepsilon)}_v}{\p v}(\overline{v},\overline{w},\overline{\mathbf{W}}).v
+ \frac{\p \phi^{(\varepsilon)}_v}{\p w}(\overline{v},\overline{w},\overline{\mathbf{W}}).w = F_v &  \text{in } \Omega \times (0,T),\\[10pt]
\displaystyle \frac{\p w}{\p t} + \delta w +
\frac{\p \phi^{(\varepsilon)}_w}{\p v}(\overline{v},\overline{w},\overline{\mathbf{W}}).v
+ \frac{\p \phi^{(\varepsilon)}_w}{\p w}(\overline{v},\overline{w},\overline{\mathbf{W}}).w = F_w &  \text{in } \Omega \times (0,T),\\[10pt]
\displaystyle \frac{\p v}{\p n} = 0 & \text{on } \p \Omega \times (0,T), \\[10pt]
(v,w)(\, \cdot \, , 0) = (v_0,w_0) & \text{in } \Omega,
\end{array} \right. \label{syslin-approx}
\end{eqnarray}
where
\begin{linenomath}\begin{equation*}
(v_0,w_0) \in \BB^{1-2/s}_{p,s}(\Omega) \times \LL^r(\Omega), \qquad
(F_v, F_w) \in \L^s(0,T; \WW^{-1,p}(\Omega)) \times \L^s(0,T;\LL^r(\Omega)).
\end{equation*}\end{linenomath}
The following result shows the surjectivity of the underlying -- non-autonomous -- linear operator.

\begin{proposition} \label{propexistlin}
There exists $\tilde{T}_0 >0$ depending on $C_{\rho}$, $(v_0,w_0)$ and $(F_v,F_w)$, such that system~\eqref{syslin-approx} admits a unique solution $(v,w) \in X^T_{s,p} \times Y^T_{s,r}$ for all $T<\tilde{T}_0$. Moreover, there exists a constant $C_{\varepsilon}(\overline{v},\overline{w}, \overline{\bW})>0$ depending only on $C_{\rho_{\varepsilon}}$ and $\overline{\bW}$ such that
\begin{eqnarray}
\|v\|_{X^T_{s,p}} + \|w\|_{Y^T_{s,r}} & \leq &
C(\overline{v},\overline{w},\overline{\bW})
\left(\|v_0\|_{\BB_{p,s}^{1-2/s}(\Omega)} + \|w_0\|_{\LL^r(\Omega)} +
\|F_v\|_{\L^s(0,T;\WW^{-1,p}(\Omega))} + \| F_w\|_{\L^s(0,T;\LL^r(\Omega))} \right).\nonumber \\
\label{estlin-na}
\end{eqnarray}
The constant $C(\overline{v},\overline{w}, \overline{\bW})$ is non-decreasing with respect to $\| \overline{\mathbf{W}}\|_{\mathcal{W}}$, and does not depend on~$\varepsilon$.
\end{proposition}

\begin{proof}
The operator $\mu : X^T_{s,p} \times Y^T_{s,r} \ni \mathrm{z} \mapsto \nabla_{\mathrm{z}} \Phi^{(\varepsilon)}(\overline{\mathrm{z}},\overline{\mathbf{W}}).\mathrm{z}$ fulfills the assumptions of Proposition~\ref{prop-locexist} for each $(\overline{\mathrm{z}},\overline{\mathbf{W}})$. From~\eqref{resLip} and~$(\mathbf{A2})$, its Lipschitz constant is controlled uniformly with respect to~$\varepsilon$ as follows
\begin{linenomath}\begin{equation}
\left\| \nabla_{\mathrm{z}} \Phi^{(\varepsilon)}(\overline{\mathrm{z}},\overline{W}) \right\|_{\mathscr{L}\left(X^T_{s,p}\times Y^T_{s,r};\L^s(0,T;\WW^{-1,p}(\Omega))\times \L^s(0,T;\LL^r(\Omega))\right)} \leq
\|\bW\|_{\mathcal{W}}^L \| \rho_{\varepsilon}'\|_{\L^{\infty}(\R)}^{\mathcal{L}_{\mathcal{W}}} \leq
 \|\bW\|_{\mathcal{W}}^L C_{\rho}^{\mathcal{L}_{\mathcal{W}}}.
\label{est-gradphi}
\end{equation}\end{linenomath}
Therefore the proof of Proposition~\ref{prop-locexist} can be repeated with such a $\mu$, and the result follows straightforwardly. The lifespan $\tilde{T}_0$ depends decreasingly on the norm of $\mathrm{z} \mapsto \nabla_{\mathrm{z}} \Phi^{(\varepsilon)}(\overline{\mathrm{z}},\overline{W})\mathrm{z}$. Thus $\tilde{T}_0$ depends only on $C_{\rho}$, $\overline{\bW}$, $(v_0,w_0)$ and $(F_v,F_w)$.
\end{proof}

\begin{remark} \label{rk-important}
The lifespan $\tilde{T}_0$ of solutions of system~\eqref{syslin-approx} is conditioned by the one of $\overline{z} = (\overline{v},\overline{w})$, which will be assumed to satisfy~\eqref{sysmain}-\eqref{sysmain2} with $\overline{\bW}$ as data. Thus, necessarily, this lifespan $\tilde{T}_0$ is smaller than or equal to the $T_0$ obtained in Theorem~\ref{th-exist0}, and consequently in what follows we will consider $T\leq \tilde{T}_0 \leq T_0$. Besides, $\tilde{T}_0$ depends on $\overline{\bW}$ basically, but in view of the constraint $\|\overline{\bW}\|^2 \leq C$ in problem~\eqref{mainpbeps}, the lifespan $\tilde{T}_0$ will actually depends only on $C_{\rho}$, $(v_0,w_0)$ and $(F_v,F_w)$.
\end{remark}

Let us introduce the adjoint system, with $\mathrm{p} = (p_v,p_w)$ as unknowns, associated with the regularized problem:
\begin{eqnarray}
\left\{ \begin{array} {ll}
\displaystyle -\frac{\p p_v}{\p t} - \nu \Delta p_v +
\frac{\p \phi_v^{\text{\textcolor{black}{$(\varepsilon)$}}}}{\p v}(\overline{v},\overline{w},\overline{\mathbf{W}})^{\ast}.p_v
+ \frac{\p \phi_v^{\text{\textcolor{black}{$(\varepsilon)$}}}}{\p w}(\overline{v},\overline{w},\overline{\mathbf{W}})^{\ast}.p_w =
 g_v &  \text{in } \Omega \times (0,T),\\[10pt]
\displaystyle -\frac{\p p_w}{\p t} + \delta p_w +
\frac{\p \phi_w^{\text{\textcolor{black}{$(\varepsilon)$}}}}{\p v}(\overline{v},\overline{w},\overline{\mathbf{W}})^{\ast}.p_v
+ \frac{\p \phi_w^{\text{\textcolor{black}{$(\varepsilon)$}}}}{\p w}(\overline{v},\overline{w},\overline{\mathbf{W}})^{\ast}.p_w =
 g_w &  \text{in } \Omega \times (0,T),\\[10pt]
\displaystyle \frac{\p p_v}{\p n} = 0 & \text{on } \p \Omega \times (0,T), \\[10pt]
(p_v,p_w)(\, \cdot \, , T) =  (p_{v,T},p_{w,T}) & \text{in } \Omega.
\end{array} \right. \label{sysadj-approx}
\end{eqnarray}
\textcolor{black}{From~\eqref{choice-func}, the right-hand-sides of this system correspond to}
\begin{linenomath}\begin{equation}
g_v = -\frac{\p R}{\p v}(\overline{v},\overline{w}), \quad
g_w = -\frac{\p R}{\p w}(\overline{v},\overline{w}), \quad
p_{v,T} = -\frac{\p R_T}{\p v}(\overline{v},\overline{w}), \quad
p_{w,T} = -\frac{\p R_T}{\p w}(\overline{v},\overline{w}). \label{good-choice}
\end{equation}\end{linenomath}
For $0<T\leq \tilde{T}_0$, we define a solution of system~\eqref{sysadj-approx} by transposition.
\begin{definition} \label{defsoladj}
Let $(p_{v,T},p_{w,T}) \in \left(\BB^{1-2/s}_{p,s}(\Omega)\right)' \times \LL^{r'}(\Omega)$ and $(g_v, g_w) \in \L^{s'}(0,T;\WW^{-1,p'}(\Omega)) \times \L^{s'}(0,T;\LL^{r'}(\Omega))$. We say that $(p_v,p_w) \in \L^{s'}(0,T;\WW^{1,p'}(\Omega)) \times \L^{s'}(0,T;\L^{r'}(\Omega))$ is solution of system~\eqref{sysadj-approx} if for all $(f_v,f_w) \in \L^s(0,T;\WW^{-1,p}(\Omega))\times \L^s(0,T;\L^r(\Omega))$ we have
\begin{eqnarray}
& & \int_0^T \left\langle  f_v;p_v\right\rangle_{\WW^{-1,p}(\Omega);\WW^{1,p'}(\Omega)}\d t
+ \int_0^T \left\langle  f_w;p_w\right\rangle_{\LL^{r'}(\Omega);\LL^{r}(\Omega)}\d t
\nonumber \\
& &  =  \left\langle p_{v,T} ; \varphi_v(T) \right\rangle_{\left(\BB^{1-2/s}_{p,s}(\Omega)\right)';\BB^{1-2/s}_{p,s}(\Omega)}
+ \left\langle p_{w,T} ; \varphi_w(T) \right\rangle_{\LL^{r'}(\Omega);\LL^{r}(\Omega)} \nonumber \\
& &\quad  + \int_0^T \left\langle  g_v;\varphi_v\right\rangle_{\WW^{-1,p'}(\Omega);\WW^{1,p}(\Omega)}\d t
+ \int_0^T \left\langle  g_w;\varphi_w\right\rangle_{\LL^{r'}(\Omega);\LL^{r}(\Omega)}\d t,
\label{eqdeftrans}
\end{eqnarray}
where $(\varphi_v,\varphi_w)$ is the solution of~\eqref{syslin-approx} with $(f_v,f_w)$ as data, namely the following linear system
\begin{eqnarray}
\left\{ \begin{array} {ll}
\displaystyle \frac{\p \varphi_v}{\p t} - \nu\Delta \varphi_v +
\frac{\p \phi_v^{\text{\textcolor{black}{$(\varepsilon)$}}}}{\p v}(\overline{v},\overline{w},\overline{\mathbf{W}}).\varphi_v
+ \frac{\p \phi_v^{\text{\textcolor{black}{$(\varepsilon)$}}}}{\p w}(\overline{v},\overline{w},\overline{\mathbf{W}}).\varphi_w = f_v &  \text{in } \Omega \times (0,T),\\[10pt]
\displaystyle \frac{\p \varphi_w}{\p t} + \delta \varphi_w +
\frac{\p \phi_w^{\text{\textcolor{black}{$(\varepsilon)$}}}}{\p v}(\overline{v},\overline{w},\overline{\mathbf{W}}).\varphi_v
+ \frac{\p \phi_w^{\text{\textcolor{black}{$(\varepsilon)$}}}}{\p w}(\overline{v},\overline{w},\overline{\mathbf{W}}).\varphi_w = f_w &  \text{in } \Omega \times (0,T),\\[10pt]
\displaystyle \frac{\p \varphi_v}{\p n} = 0 & \text{on } \p \Omega \times (0,T), \\[10pt]
(\varphi_v,\varphi_w)(\, \cdot \, , 0) = (0,0) & \text{in } \Omega.
\end{array} \right. \label{syslintrans}
\end{eqnarray}
\end{definition}
Note that the existence of a solution $(\varphi_v,\varphi_w)$ to system~\eqref{syslintrans} is guaranteed by Proposition~\ref{propexistlin}. We state now an existence result for the adjoint system~\eqref{sysadj-approx}.

\begin{proposition}
Let be $T\in (0, \tilde{T}_0)$. For all
\begin{eqnarray*}
 (p_{v,T},p_{w,T}) \in \left(\BB^{1-2/s}_{p,s}(\Omega)\right)' \times \LL^{r'}(\Omega) & &
(g_v,g_w) \in \L^{s'}(0,T;\WW^{-1,p'}(\Omega)) \times \L^{s'}(0,T;\LL^{r'}(\Omega)),
\end{eqnarray*}
system~\eqref{sysadj-approx} admits a unique solution $(p_v,p_w)$ in the sense of Definition~\ref{defsoladj}, satisfying
\begin{eqnarray}
\|p_v\|_{\L^{s'}(0,T;\WW^{1,p'}(\Omega))} + \|p_w \|_{\L^{s'}(0,T;\LL^{r'}(\Omega))}
& \leq & C(\overline{v},\overline{w}, \overline{\bW}) \left(\|p_{v,T} \|_{\left(\BB^{1-2/s}_{p,s}(\Omega)\right)'}
+ \|p_{w,T}\|_{\LL^{r'}(\Omega)}\right. \nonumber \\
& & \left.
\qquad + \| g_v \|_{\L^{s'}(0,T;\WW^{-1,p'}(\Omega))}
+ \| g_w \|_{\L^{s'}(0,T;\LL^{r'}(\Omega))}
\right), \label{estimate-adjoint}
\end{eqnarray}
where the constant $C(\overline{v},\overline{w}, \overline{\bW}) > 0$ is the one which appears in~\eqref{estlin-na}.
\end{proposition}

\begin{proof}
Denote by $\Lambda(\overline{v},\overline{w},\overline{\mathbf{W}})$ the mapping defined by
\begin{eqnarray*}
\Lambda(\overline{v},\overline{w},\overline{\mathbf{W}}):(f_v,f_w) & \mapsto  &
\big( (\varphi_v, \varphi_w), (\varphi_v(T),\varphi_w(T)) \big),
\end{eqnarray*}
where $(\varphi_v,\varphi_w)$ is the solution of system~\eqref{syslintrans}. From Proposition~\ref{propexistlin}, the operator $\Lambda(\overline{v},\overline{w},\overline{\mathbf{W}})$ is bounded from $\L^s(0,T;\WW^{-1,p}(\Omega))\times \L^s(0,T;\L^r(\Omega))$ into $\L^s(0,T; \WW^{1,p}(\Omega)) \times \L^s(0,T;\LL^r(\Omega)) \times \BB^{1-2/s}_{p,s}(\Omega) \times \LL^{r}(\Omega)$. So the adjoint operator $\Lambda(\overline{v},\overline{w},\overline{\mathbf{W}})^\ast$ is bounded from $\L^{s'}(0,T; \WW^{-1,p'}(\Omega)) \times \L^{s'}(0,T;\LL^{r'}(\Omega)) \times\left(\BB^{1-2/s}_{p,s}(\Omega)\right)' \times \LL^{r'}(\Omega)$ into $\L^{s'}(0,T;\WW^{1,p'}(\Omega))\times \L^{s'}(0,T;\L^{r'}(\Omega))$, with the same continuity constant (due to the Hahn-Banach theorem), namely the one which appears in estimate~\eqref{estlin-na}. Now by setting $(p_v,p_w) = \Lambda(\overline{v},\overline{w},\overline{\bW})^\ast\big((g_v,g_w),(p_{v,T},p_{w,T})\big)$, we can verify that $(p_v,p_w) \in \L^{s'}(0,T;\WW^{1,p'}(\Omega))\times \L^{s'}(0,T;\L^{r'}(\Omega))$ satisfies system~\eqref{sysadj-approx}. Uniqueness is due to the fact that for $(g_v,g_w) = (0,0)$ and $(p_{v,T},p_{w,T})=(0,0)$, the corresponding solution $(p_v,p_w)$ of~\eqref{sysadj-approx} in the sense of Definition~\ref{defsoladj} satisfies~\eqref{eqdeftrans}, from which we deduce $(p_v,p_w) = (0,0)$ in $\L^{s'}(0,T;\WW^{1,p'}(\Omega))\times \L^{s'}(0,T;\L^{r'}(\Omega))$.
\end{proof}

\subsection{On the control-to-state mapping for the regularized problem}

For $0< T < \tilde{T}_0$, we define the control-to-state mapping for problem~\ref{mainpbeps} as
\begin{linenomath}\begin{equation*}
\begin{array} {rrcl}
 \mathbb{S}_{\varepsilon} : & \mathcal{W} & \rightarrow & X^T_{s,p} \times Y^T_{s,r} \\
& \mathbf{W} & \mapsto & \mathrm{z}
\end{array}
\end{equation*}\end{linenomath}
where $\mathrm{z} = (v,w)$ is the solution of~\eqref{sysmain}-\eqref{sysmain2} with $\rho_{\varepsilon}$ as activation function. From Theorem~\ref{th-exist0}, the mapping $\mathbb{S}_{\varepsilon}$ is well-defined. Furthermore, we have:

\begin{proposition}
The mapping $\mathbb{S}_{\varepsilon}$ is of class $C^1$ on $\mathcal{W}$.
\end{proposition}

\begin{proof}
Define the mapping
\begin{eqnarray}
\begin{array} {rrcl}
e: & \left(X^T_{s,p} \times Y^T_{s,r} \right)\times  \mathcal{W} & \rightarrow &
\L^s(0,T; \WW^{-1,p}(\Omega)) \times \L^s(0,T;\LL^r(\Omega)) \times \BB^{1-2/s}_{p,s}(\Omega) \times \LL^r(\Omega) \\
 & (\mathrm{z} = (v,w), \mathbf{W}) & \mapsto & \left(\begin{array} {c}
\displaystyle \frac{\p v}{\p t} - \nu \Delta v + \phi^{(\varepsilon)}_v(v,w) - f_v \\[10pt]
\displaystyle \frac{\p w}{\p t} + \delta w + \phi^{(\varepsilon)}_w(v,w) - f_w \\
v(\cdot, 0) - v_0 \\
w(\cdot, 0) -w_0
\end{array}\right).
\end{array} \label{eq-constraint}
\end{eqnarray}
Note that $e(\mathbb{S}_{\varepsilon}(\mathbf{W}), \mathbf{W}) = 0$ for all $\mathbf{W} \in \mathcal{W}$. Besides, due to the regularity of activation functions $\rho_{\varepsilon}$, the mapping $e$ is of class $C^1$, and from Proposition~\ref{propexistlin} the linear mapping $e_{\mathrm{z}}((\overline{v},\overline{w}), \overline{\mathbf{W}})$ is surjective for all $(\overline{v},\overline{w}, \overline{\mathbf{W}}) \in X^T_{s,p} \times Y^T_{s,r} \times  \mathcal{W}$. Then the result follows from the implicit function theorem.
\end{proof}

\subsection{Necessary optimality conditions for the regularized problem}

A solution to problem~\eqref{mainpbeps} can be seen as a saddle-point of the following Lagrangian functional:
\begin{eqnarray}
\mathscr{L}(\mathbf{W}, \lambda) & = & J(\mathbb{S}_{\varepsilon}(\mathbf{W}), \mathbf{W}) - \lambda\left(\|\mathbf{W}\|^2_{\mathcal{W}} - C \right).
\end{eqnarray}

\begin{proposition} \label{propopteps}
Let $\overline{\mathbf{W}}$ be an optimal solution of~\eqref{mainpbeps}. Then there exists $\overline{\lambda } \geq 0$ such that for all $1 \leq \ell \leq L $
\begin{eqnarray}
\int_0^T \int_\Omega
\frac{\p \Phi^{(\varepsilon)}}{\p W_\ell}(\overline{\mathrm{z}},\overline{\mathbf{W}})^\ast \overline{\mathrm{p}} \, \d \Omega  - 2\overline{\lambda} \, \overline{W}_{\ell} =0, \label{optcondeps1}\\
\overline{\lambda}\left(\|\overline{\mathbf{W}}\|^2_{\mathcal{W}} - C\right) = 0,\label{optcondeps2}
\end{eqnarray}
where $\overline{\mathrm{z}}$ is solution of~\eqref{sysmain}-\eqref{sysmain2} with $\overline{\mathbf{W}}$ as data, and $\overline{\mathrm{p}}$ solution of~\eqref{sysadj-approx}-\eqref{good-choice}, with $\overline{\mathrm{z}} = (\overline{v},\overline{w})$ and $\overline{\mathbf{W}}$ as data.
\end{proposition}

\begin{proof}
Note that the functional~$\mathscr{L}$ is of class $C^1$. Let us calculate its first-order derivatives. Before, we denote
\begin{eqnarray*}
\mathcal{S} = X^T_{s,p} \times Y^T_{s,r}, & &
\mathcal{V} = \left(\L^s(0,T;\WW^{-1,p}(\Omega))\times \L^s(0,T;\L^r(\Omega))\right)
 \times \left(\BB^{1-2/s}_{s,p}(\Omega) \times \LL^r(\Omega)\right),
\end{eqnarray*}
and by using the derivatives of the mapping $e$ defined in~\eqref{eq-constraint}, we notice that
\begin{eqnarray}
 e(\mathbb{S}_{\varepsilon}(\bW), \bW) = 0 & \Rightarrow &
\left\langle e'_{\mathrm{z}}(\mathbb{S}_{\varepsilon}(\bW), \bW) ; \mathbb{S}_{\varepsilon}'(\bW).\tilde{\bW} \right\rangle_{\mathcal{S}',\mathcal{S}}
+ \left\langle e'_{\bW}(\mathbb{S}_{\varepsilon}(\bW), \bW);\tilde{\bW} \right\rangle_{\mathcal{W}} = 0, \label{useful1} \\
J'_{\mathrm{z}}(\mathrm{z}, \bW) & = & -e'_{\mathrm{z}}(\mathrm{z},\bW)^{\ast}(\mathrm{p},\mathrm{p}(\cdot,T)), \label{useful3}
\end{eqnarray}
where $\mathrm{p}$ satisfies~\eqref{sysadj-approx}-\eqref{good-choice} with $(\mathrm{z} = \mathbb{S}_{\varepsilon}(\bW), \bW)$ as data. Thus we have
\begin{eqnarray*}
\mathscr{L}_{\mathbf{W}}(\bW,\lambda).\tilde{\bW} & = &
\left\langle J'_{\mathrm{z}}(\mathbb{S}_{\varepsilon}(\bW), \bW) ; \mathbb{S}'_{\varepsilon}(\bW).\tilde{\bW} \right\rangle_{\mathcal{S}',\mathcal{S}} + J'_{\bW}(\mathbb{S}_{\varepsilon}(\bW), \bW).\tilde{\bW} - 2\lambda \langle \bW ; \tilde{\bW} \rangle_{\mathcal{W}}\\
& = & 
-2\lambda\langle \bW ; \tilde{\bW} \rangle_{\mathcal{W}}
- \left\langle e'_{\mathrm{z}}(\mathbb{S}_{\varepsilon}(\bW),\bW)^{\ast}(\mathrm{p},\mathrm{p}(\cdot,T)) ; \mathbb{S}_{\varepsilon}'(\bW). \tilde{\bW} \right\rangle_{\mathcal{S}',\mathcal{S}}
\end{eqnarray*}
where we used~\eqref{useful3}. Next, using~\eqref{useful1}, we get
\begin{eqnarray*}
\mathscr{L}_{\mathbf{W}}(\bW,\lambda).\tilde{\bW} & = & 
-2\lambda\langle \bW ; \tilde{\bW} \rangle_{\mathcal{W}}
- \left\langle (\mathrm{p},\mathrm{p}(\cdot,T)) ; e'_{\mathrm{z}}(\mathbb{S}_{\varepsilon}(\bW),\bW).\left(\mathbb{S}_{\varepsilon}'(\bW). \tilde{\bW}\right) \right\rangle_{\mathcal{V}';\mathcal{V}} \\
& = & 
-2\lambda\langle \bW ; \tilde{\bW} \rangle_{\mathcal{W}}
+ \left\langle (\mathrm{p},\mathrm{p}(\cdot,T)) ; e'_{\bW}(\mathbb{S}_{\varepsilon}(\bW),\bW). \tilde{\bW} \right\rangle_{\mathcal{V}';\mathcal{V}} \\
& = & 
-2\lambda\langle \bW ; \tilde{\bW} \rangle_{\mathcal{W}}
+ \left\langle e'_{\bW}(\mathbb{S}_{\varepsilon}(\bW),\bW)^{\ast}(\mathrm{p},\mathrm{p}(\cdot,T)) ; \tilde{\bW} \right\rangle_{\mathcal{W}}.
\end{eqnarray*}
Note that $e'_{\bW}(\mathbb{S}_{\varepsilon}(\bW),\bW) = \big(\nabla_{\bW} \Phi^{(\varepsilon)}(\mathbb{S}_{\varepsilon}(\bW),\bW), 0\big)$ in $\mathscr{L}(\mathcal{W}, \mathcal{V})$. Then the result follows from the Karush-Kuhn-Tucker conditions \textcolor{black}{(see for instance~\cite[section~2.8, p.~63]{Troeltzsch})}.
\end{proof}

\section{Passing to the limit} \label{sec-passing}

The purpose of this section is to derive optimality conditions for problem~\eqref{mainpb}, from those obtained in proposition~\ref{propopteps} for problem~\eqref{mainpbeps}. First we need to verify that solutions of~\eqref{mainpb} can be approached by solutions of~\eqref{mainpbeps}. Next, we need boundedness results on dual variables that appear in~\eqref{optcondeps1}-\eqref{optcondeps2}. The main result is then given in section~\ref{sec-final}.

\subsection{Convergence of minimizers}

We start with proving the convergence of the control-to-state mapping.

\begin{lemma} \label{lemma-cvS}
The sequence $(\mathbb{S}_{\varepsilon})_{\varepsilon}$ converges uniformly towards $\mathbb{S}$ on $\mathcal{W}$, with respect to the norm $X^T_{s,p}\times Y^T_{s,r}$.
\end{lemma}

\begin{proof}
Let be $\mathbf{W} \in \mathcal{W}$, and set $\mathrm{z}_{\varepsilon}=\mathbb{S}_{\varepsilon}(\mathbf{W})$ and  $\mathrm{z}=\mathbb{S}(\mathbf{W})$. From~\eqref{est-lin}, the difference $\mathrm{z}_{\varepsilon} - \mathrm{z}$ satisfies the following estimate
\begin{eqnarray} \label{myesti}
\| \mathrm{z}_{\varepsilon} - \mathrm{z} \|_{X^T_{s,p}\times Y^T_{s,r}}
& \leq & C \| \Phi^{(\varepsilon)}(\mathrm{z}_{\varepsilon},\mathbf{W}) - \Phi(\mathrm{z},\mathbf{W}) \|_{\L^s(0,T;\WW^{-1,p}(\Omega)) \times \L^s(0,T;\LL^{r}(\Omega))}.
\end{eqnarray}
In the right-hand-side of the estimate above, we decompose
\begin{eqnarray*}
\Phi^{(\varepsilon)}(\mathrm{z}_{\varepsilon},\mathbf{W}) - \Phi(\mathrm{z},\mathbf{W}) =
\big(\Phi_{\varepsilon}(\mathrm{z}_{\varepsilon},\mathbf{W}) - \Phi^{(\varepsilon)}(\mathrm{z},\mathbf{W})\big)
+ \big(\Phi^{(\varepsilon)}(\mathrm{z},\mathbf{W}) - \Phi(\mathrm{z},\mathbf{W})\big).
\end{eqnarray*}
First, from Lemma~\ref{lemma-mu} we estimate
\begin{eqnarray*}
\|\Phi^{(\varepsilon)}(\mathrm{z}_{\varepsilon},\mathbf{W}) - \Phi^{(\varepsilon)}(\mathrm{z},\mathbf{W})\|_{\L^s(0,T;\WW^{-1,p}(\Omega)) \times \L^s(0,T;\LL^{r}(\Omega))} \leq CC_{\Phi^{(\varepsilon)}}T^{\alpha}
\| \mathrm{z} - \mathrm{z}_{\varepsilon} \|_{X^T_{s,p}\times Y^T_{s,r}},
\end{eqnarray*}
where $\alpha >0$ is such that $T^{\alpha} \geq \max\left(T^{1/2}, T^{1/(2s)}\right)\max\left(1,T^{1/s}\right)$, and where the Lipschitz constant $C_{\Phi^{(\varepsilon)}}$ of $\Phi^{(\varepsilon)}$ is controlled by $C_{\rho}^{\mathcal{L}_{\mathcal{W}}}$ (see~\eqref{resLip}), due to Assumption~$(\mathbf{A2})$. 
Therefore, we can thus reduce~\eqref{myesti} to
\begin{eqnarray*}
\| \mathrm{z}_{\varepsilon} - \mathrm{z} \|_{X^T_{s,p}\times Y^T_{s,r}}
& \leq & C \| \Phi^{(\varepsilon)}(\mathrm{z},\mathbf{W}) - \Phi(\mathrm{z},\mathbf{W}) \|_{\L^s(0,T;\WW^{-1,p}(\Omega)) \times \L^s(0,T;\LL^{r}(\Omega))}.
\end{eqnarray*}
Recall from Assumption~$(\mathbf{A2})$ that $\rho_{\varepsilon}$ converges uniformly in $C(\R;\R)$ towards $\rho$, so by composition we have $ \Phi^{(\varepsilon)}(\mathrm{z},\cdot) \rightarrow \Phi(\mathrm{z},\cdot)$ uniformly on $\mathcal{W}$ with respect to $C(\mathcal{W};\R^2)$, and {\it a fortiori} for the norm of $\R^2$ replaced by the one of the right-hand-side of the estimate above, which concludes the proof.
\end{proof}

\textcolor{black}{The convergence of minimizers relies on the notion of $\Gamma$-convergence (see~\cite{DalMaso}), and more specifically on the fundamental theorem of $\Gamma$-convergence, stating that for a sequence of functionals that $\Gamma$-converges, the minimizers of the functionals of this sequence do converge towards minimizers of the limit functional.}
\begin{proposition}
Minimizers of problem~\eqref{mainpbeps} converge towards minimizers of problem~\eqref{mainpb} when $\varepsilon$ goes to $0$, up to extraction of a subsequence.
\end{proposition}

\begin{proof}
Consider the sequence of functionals
\begin{linenomath}\begin{equation*}
\tilde{J}_{\varepsilon}: \mathbf{W} \mapsto J(\mathbb{S}_{\varepsilon}(\mathbf{W}), \mathbf{W}).
\end{equation*}\end{linenomath}
\textcolor{black}{Let us prove} that the sequence $(\tilde{J}_{\varepsilon})_{\varepsilon >0}$ $\Gamma$-converges towards $\tilde{J} := J(\mathbb{S}(\cdot),\cdot)$. Since from Lemma~\ref{lemma-cvS} the convergence $\mathbb{S}_{\varepsilon} \rightarrow \mathbb{S}$ is uniform on $\mathcal{W}$, the mapping $\mathbb{S}$ is continuous. Since $J$ is continuous, it follows that $(\tilde{J}_{\varepsilon})_{\varepsilon}$ converges uniformly towards $\tilde{J}$ on $\mathcal{W}$. Therefore $(\tilde{J}_{\varepsilon})_{\varepsilon}$ converges towards the lower semi-continuous version of $\tilde{J} := J(\mathbb{S}(\cdot),\cdot)$, that is actually $\tilde{J}$ itself, since by assumption $J$ itself is lower semi-continuous. Then from the fundamental theorem of $\Gamma$-convergence (see~\cite[Corollary~7.20, page~81]{DalMaso}), minimizers of $\tilde{J}_{\varepsilon}$ converge towards minimizers of $\tilde{J}$, and thus the proof is complete.
\end{proof}

\subsection{Boundedness of dual variables}

Following the idea utilized in~\cite{Meyer1}, system~\eqref{sysadj-approx}-\eqref{good-choice} can be rewritten by introducing a variable $\bmu$ as follows
\begin{eqnarray}
\left\{ \begin{array} {ll}
\displaystyle -\frac{\p \mathrm{p}}{\p t} + \mathrm{A}_\delta \mathrm{p} +
\bmu = \left( \begin{matrix} g_v \\ g_w \end{matrix}\right) &  \text{in } \Omega \times (0,T),\\[10pt]
\bmu =  \nabla_\mathrm{z} \Phi^{(\varepsilon)}(\mathrm{z},\mathbf{W} )^\ast \mathrm{p} & \text{in } \Omega \times (0,T),\\[10pt]
\displaystyle \frac{\p p_v}{\p n} = 0 & \text{on } \p \Omega \times (0,T), \\[10pt]
(p_v,p_w)(\, \cdot \, , T) =  (p_{v,T},p_{w,T}) & \text{in } \Omega,
\end{array} \right. \label{sysadj-approx2}
\end{eqnarray}
where $\mathrm{A}_\delta = \displaystyle \left( \begin{matrix} -\nu\Delta & 0 \\ 0 & \delta \, \Id \end{matrix} \right)$, and
\begin{eqnarray}
g_v = -\frac{\p R}{\p \mathrm{z}_1}(\mathrm{z}), \quad
g_w = -\frac{\p R}{\p \mathrm{z}_2}(\mathrm{z}), \quad
p_{v,T} = -\frac{\p R_T}{\p \mathrm{z}_1}(\mathrm{z}), \quad
p_{w,T} = -\frac{\p R_T}{\p \mathrm{z}_2}(\mathrm{z}). \label{4rhs_adj}
\end{eqnarray}
Recall that $R,\, R_T: X^T_{s,p} \times Y_{s,r}^T \rightarrow \R$ are assumed to be~$C^1$. From now we consider the triplet $(\overline{\bW}_{\varepsilon}, \overline{\bmu}_{\varepsilon}, \overline{\lambda}_{\varepsilon})$ for denoting a solution to problem~\eqref{mainpbeps}. Note that from the optimality conditions of Proposition~\ref{propopteps}, when the optimal $\overline{\bW}_{\varepsilon}$ is given, the value of the corresponding optimal $\overline{\bmu}_{\varepsilon}$ is entirely determined as
\begin{linenomath}\begin{equation*}
\overline{\bmu}_{\varepsilon} = \nabla_\mathrm{z} \Phi^{(\varepsilon)}(\overline{\mathrm{z}}_{\varepsilon},\overline{\bW}_{\varepsilon})^{\ast}\overline{\mathrm{p}}_{\varepsilon},
\end{equation*}\end{linenomath}
where $\overline{\mathrm{z}}_{\varepsilon}$ is solution of~\eqref{sysmain}-\eqref{sysmain2} with $\overline{\mathbf{W}}_{\varepsilon}$ as data, and $\overline{\mathrm{p}}_{\varepsilon}$ solution of~\eqref{sysadj-approx}-\eqref{good-choice}, with $\overline{\mathrm{z}}_{\varepsilon} = (\overline{v}_{\varepsilon},\overline{w}_{\varepsilon})$ and $\overline{\mathbf{W}}_{\varepsilon}$ as data. Let us give $\varepsilon$-independent bounds for these quantities.

\begin{proposition} \label{prop-bound}
Let $(\overline{\mathbf{W}}_\varepsilon, \overline{\bmu}_\varepsilon, \overline{\lambda}_\varepsilon)$ be optimal for problem~\eqref{mainpbeps}. Denote by~$\overline{\mathrm{z}}_{\varepsilon}$ the solution of~\eqref{sysmain}-\eqref{sysmain2} with $\rho_{\varepsilon}$ as activation function and $\overline{\bW}_{\varepsilon}$ as weights. Denote by $\overline{\mathrm{p}}_{\varepsilon}$ the solution of~\eqref{sysadj-approx2} with $(\overline{\mathrm{z}}_{\varepsilon}, \overline{\bW}_{\varepsilon})$ as data, and $\overline{\bmu}_{\varepsilon} = \nabla_{\mathrm{z}} \Phi^{(\varepsilon)}(\overline{\mathrm{z}}_{\varepsilon}, \overline{\bW}_{\varepsilon})^{\ast}\overline{\mathrm{p}}_{\varepsilon}$. Then there exists a constant $C>0$ independent of $\varepsilon$ such that
\begin{eqnarray}
\| \overline{\mathrm{z}}_{\varepsilon} \|_{X^T_{s,p} \times Y^T_{s,r}} & \leq & C,
\label{estimate_z}\\	
\| \overline{\mathrm{p}}_\varepsilon \|_{\left(\L^{s'}(0,T;\WW^{1,p'}(\Omega))\cap \W^{1,s'}(0,T;\WW^{-1,p'}(\Omega))\right) \times \W^{1,s'}(0,T;\LL^{r'}(\Omega))} & \leq & C,
\label{estimate_p} \\
\| \overline{\bmu}_\varepsilon \|_{\L^{s'}(0,T;\WW^{-1,p'}(\Omega))\times \L^{s'}(0,T;\LL^{r'}(\Omega))} & \leq & C.
\label{estimate_mu}
\end{eqnarray}
The constant $C>0$ depends only on~$(v_0,w_0)$ and~$(f_v,f_w)$.

\end{proposition}

\begin{proof}
Keep in mind that the constraint $\|\overline{\bW}\|_{\varepsilon} \leq C$ is satisfied, and that from~$(\mathbf{A2})$ the Lipschitz constant of $\rho_{\varepsilon}$ is controlled by the one of $\rho$. Then~\eqref{estimate_z} is obtained from~\eqref{estexistlocal}. Estimate~\eqref{estimate_p} follows from~\eqref{estimate-adjoint} and~\eqref{good-choice}, combined with~\eqref{estimate_z}. Finally,~\eqref{estimate_mu} is due to~\eqref{est-gradphi} combined with~\eqref{estimate_p}.
\end{proof}

\subsection{Optimality conditions for non-smooth activation functions} \label{sec-final}

Like in~\cite{Meyer1}, we derive a wellposedness result for a relaxed version of system~\eqref{sysadj-approx2}.

\begin{proposition} \label{prop10}
There exists an adjoint-state $\mathrm{p}$ and a multiplier $\bmu$ which are solutions of the following system
\begin{eqnarray}
\left\{ \begin{array} {ll}
\displaystyle -\frac{\p \mathrm{p}}{\p t} + \mathrm{A}_{\delta}^{\ast} \mathrm{\mathrm{p}} +
\bmu = \left( \begin{matrix} g_v \\ g_w \end{matrix}\right) &  \text{in } \Omega \times (0,T),\\[10pt]
\bmu \in  \displaystyle
\left(\left(\mathrlap{\prod_{\ell = 1}^{L}}{\hspace*{-0.2pt}\longrightarrow} A_{\ell} \right)^{\ast} \left[(\rho'_{\mathrm{min}})^{L-1}, (\rho'_{\mathrm{max}})^{L-1}\right]\right) \mathrm{p}
& \text{in } \Omega \times (0,T),\\[10pt]
\displaystyle \frac{\p p_v}{\p n} = 0 & \text{on } \p \Omega \times (0,T), \\[10pt]
(p_v,p_w)(\, \cdot \, , T) =  (p_{v,T},p_{w,T}) & \text{in } \Omega,
\end{array} \right. \label{syssubdiff}
\end{eqnarray}
where the notation used in the second equation above is described in the following proof.
\end{proposition}

\begin{proof}
Let us explain the second line of~\eqref{syssubdiff}. Recall that from~\eqref{grad_nn} we have with matrix notation
\begin{eqnarray*}
\left(\nabla_{\mathrm{z}}\Phi_L^{(\varepsilon)}\right)_{ij} & = & \left(A_L\mathrlap{\prod_{\ell = 1}^{L-1}}{\hspace*{-0.2pt}\longrightarrow} \rho_{\varepsilon}'\big(\Phi^{(\varepsilon)}_\ell\big) A_{\ell}\right)_{ij}\\
& = &
\sum_{k_1} (A_L)_{ik_1} \left(\sum_{k_{3/2},k_2, \dots , k_{L-3/2},k_{L-1}}
\prod_{\ell = 1}^{L-1}
\rho_{\varepsilon}'(\Phi^{(\varepsilon)}_{\ell})_{k_{\ell}k_{\ell+1/2}} (A_{\ell})_{k_{\ell +1/2}k_{\ell+1}}\right) \quad \text{with $k_L := j$}.
\end{eqnarray*}
Since $\rho(x)_i = \rho(x_i)$, the gradient of $x\mapsto \rho(x)$ is represented by a diagonal matrix, such that $\rho'(x)_{k_ak_b} = \delta_{k_ak_b}\rho'(x)_{k_ak_b} = \rho'(x_{k_a})$ if $k_a = k_b$, $0$ otherwise. Thus we get
\begin{linenomath}\begin{equation*}
\left(\nabla_{\mathrm{z}} \Phi^{(\varepsilon)}_L\right)_{ij}  =
\sum_{k_1, \dots , k_{L-1}} (A_L)_{ik_1} \prod_{\ell=1}^{L-1} \rho'_{\varepsilon}(\Phi^{(\varepsilon)}_{k_{\ell}}) (A_{\ell})_{k_{\ell}k_{\ell+1}} \quad \text{with } k_L =j.
\end{equation*}\end{linenomath}
From assumption~$(\mathbf{A2})'$ we deduce
\begin{linenomath}\begin{equation*}
\left(\nabla_{\mathrm{z}} \Phi^{(\varepsilon)}_L\right)_{ij}  \in
\sum_{k_1, \dots k_{L-1}} (A_L)_{ik_1} \left(\prod_{\ell=1}^{L-1} (A_{\ell})_{k_{\ell}k_{\ell+1}}\right) \left[(\rho'_{\mathrm{min}})^{L-1}, (\rho'_{\mathrm{max}})^{L-1}\right] \quad \text{with } k_0 := i,
\end{equation*}\end{linenomath}
where we define the multiplication $a[r^-,r^+] := [\min(ar^-,ar^+), \max(ar^-,ar^+)]$ for all $a,r^-,r^+ \in R$ with $r^- \leq r^+$, and use the Minkowski addition for sets. Thus we have
\begin{linenomath}\begin{equation*}
\nabla_{\mathrm{z}} \Phi^{(\varepsilon)}_L  \in
 \left(\mathrlap{\prod_{\ell = 1}^{L}}{\hspace*{-0.2pt}\longrightarrow}
A_{\ell}\right)\left[(\rho'_{\mathrm{min}})^{L-1}, (\rho'_{\mathrm{max}})^{L-1}\right],
\label{almost}
\end{equation*}\end{linenomath}
leading to the expression used in the second equation of~\eqref{syssubdiff}. The result is very similar to~\cite[Theorem~6.7]{Meyer1}, so its proof is not repeated.
\end{proof}

Now we are in position to provide optimality conditions for problem~\eqref{mainpb} in the case of non-smooth activation functions.

\begin{theorem}
Let $\overline{\mathbf{W}}$ be an optimal solution of~\eqref{mainpb}. Then there exists $\overline{\lambda } \geq 0$ such that for all $1 \leq \ell \leq L-1 $
\begin{linenomath}\begin{equation}
\left.\begin{array} {r}
\displaystyle\left[(\rho'_{\mathrm{min}})^{L-\ell}, (\rho'_{\mathrm{max}})^{L-\ell}\right]\int_0^T \displaystyle\int_\Omega
\left(\left(\mathrlap{\prod_{k = \ell}^{L-1}}{\hspace*{-0.2pt}\longrightarrow} A_{k+1} \right)\rho(\Phi_{\ell-1}(\overline{\mathrm{z}},\overline{\mathbf{W}}))\right)^{\ast}  \overline{\mathrm{p}} \, \d \Omega  - 2\overline{\lambda} \, \overline{W}_{\ell} \ni 0, \\[10pt]
\displaystyle \int_0^T \displaystyle\int_\Omega
\rho(\Phi_{L-1}(\overline{\mathrm{z}},\overline{\mathbf{W}})) \cdot \overline{\mathrm{p}} \, \d \Omega  - 2\overline{\lambda} \, \overline{W}_{L} = 0, \\[10pt]
\overline{\lambda}\left(\|\overline{\mathbf{W}}\|^2_{\mathcal{W}} - C\right) = 0,
\end{array} \right. \label{optcondsub}
\end{equation}\end{linenomath}
where $\overline{\mathrm{p}}$ satisfies~\eqref{syssubdiff} with right-hand-side given as in~\eqref{4rhs_adj}, with $\mathrm{z} = \overline{\mathrm{z}}$ solution of~\eqref{sysmain}-\eqref{sysmain2} with $\overline{\mathbf{W}}$ as data.
\end{theorem}



\begin{proof}
Consider the sequence -- indexed by $\varepsilon$ -- of optimal solutions to problem~\ref{mainpbeps} provided by Proposition~\ref{propopteps}. Note that the optimality conditions given by~\eqref{syssubdiff}-\eqref{optcondsub} are satisfied by $(\overline{\bW}_{\varepsilon}, \overline{\bmu}_{\varepsilon}, \overline{\lambda}_{\varepsilon})$. We have in particular
\begin{linenomath}\begin{equation*}
\overline{\mu}_{\varepsilon} \in \left\{ \nabla_{\mathrm{z}}\Phi^{(\varepsilon)}(\overline{\mathrm{z}}_{\varepsilon}, \overline{\mathrm{W}}_{\varepsilon})^{\ast}\overline{\mathrm{p}}_{\varepsilon} \right\} \subset
\displaystyle
\left(\left(\mathrlap{\prod_{\ell = 1}^{L}}{\hspace*{-0.2pt}\longrightarrow} A_{\ell} \right)^{\ast} \left[(\rho'_{\mathrm{min}})^{L-1}, (\rho'_{\mathrm{max}})^{L-1}\right]\right) \overline{\mathrm{p}}_{\varepsilon},
\end{equation*}\end{linenomath}
and on the order side from~\eqref{formula_phiW} we deduce, like in proof of Proposition~\ref{prop10},
\begin{eqnarray*}
\frac{\p \Phi^{(\varepsilon)}}{\p W_\ell}(\overline{\mathrm{z}}_{\varepsilon},\overline{\mathbf{W}}_{\varepsilon})
& \in &
\left(\mathrlap{\prod_{k = \ell}^{L-1}}{\hspace*{-0.2pt}\longrightarrow} A_{k+1}
[ \rho'_{\mathrm{min}}, \rho'_{\mathrm{max}}]
\right)\rho(\Phi_{\ell-1}(\overline{\mathrm{z}}_{\varepsilon},\overline{\mathbf{W}}_{\varepsilon}))\\
& \in & 
[ (\rho'_{\mathrm{min}})^{L-\ell}, (\rho'_{\mathrm{max}})^{L-\ell}] \left(\mathrlap{\prod_{k = \ell}^{L-1}}{\hspace*{-0.2pt}\longrightarrow} A_{k+1}
\right)\rho(\Phi_{\ell-1}(\overline{\mathrm{z}}_{\varepsilon},\overline{\mathbf{W}}_{\varepsilon})).
\end{eqnarray*}
Besides, from Proposition~\ref{prop-bound} the different variables indexed by~$\varepsilon$ are bounded, and so up to extraction we can assume that they converge weakly. Passing to the limit in the linear system~\eqref{sysadj-approx2} -- with $(\mathrm{z},\mathrm{p}, \bmu)$ replaced by $(\mathrm{z}_{\varepsilon},\mathrm{p}_{\varepsilon}, \bmu_{\varepsilon})$ -- is straightforward. Passing to the limit in the nonlinear terms of the state equation can be achieved like in the proof of Proposition~\ref{prop-exist-min}, namely via strong convergence.
\end{proof}

\begin{remark}
When passing to the limit in the second equation of~\eqref{sysadj-approx2}, we are not able to get
\begin{linenomath}\begin{equation*}
\bmu \in [\nabla_{\mathrm{z}} \Phi]^{\ast} \mathrm{p} \quad \text{in } \Omega \times (0,T).
\end{equation*}\end{linenomath}
Indeed, even when the set of points where $\rho$ is not differentiable is only countable, the values of $\rho'_{\varepsilon}$ may possibly be such that those of $\nabla_{\mathrm{z}} \Phi^{(\varepsilon)}$ oscillate \textcolor{black}{outside the bounds} of $\nabla_{\mathrm{z}} \Phi$ when $\varepsilon$ goes to zero, which constitutes a difficulty that goes beyond the scope of the present work. Therefore we consider only inclusions as given in the second equation of~\eqref{syssubdiff} and the first equation of~\eqref{optcondsub}.
\end{remark}

\section{Numerical illustrations} \label{sec-num}

\textcolor{black}{We test the feasibility of our approach by designing neural networks that approximate the FitzHugh-Nagumo model and the Aliev-Panfilov model, with different types of activation functions.}

\subsection{Approximating the FitzHugh-Nagumo model}
\textcolor{black}{First, we try to re-discover the FitzHugh-Nagumo model.}
More precisely, we create data corresponding to the state solution of the FitzHugh-Nagumo model, and design a neural network that minimizes a distance between these data and the outputs of the monodomain model with the neural network as nonlinearity. For simplicity, we consider the ODE version of the monodomain model:
\begin{linenomath}\begin{equation}
\left\{ \begin{array} {ll}
\displaystyle \dot{v} + \phi_v(v,w) = f_v &
\text{in } (0,T),\\[10pt]
\displaystyle \dot{w} + \delta w + \phi_w(v,w) = f_w &
\text{in } (0,T),\\[10pt]
(v,w)(0) = (v_0,w_0)\in \R^2. &
\end{array} \right. \label{ODE}
\end{equation}\end{linenomath}
The solutions $\mathrm{z} = (v,w)$ of~\eqref{ODE} correspond to average values over $\Omega$ of the solutions of system~\eqref{sysmain}. Data are pairs $\left\{(\mathrm{z}^{(0)}_{\mathrm{data},k},\mathrm{z}_{\mathrm{data},k}); 1\leq k \leq K\right\}$, where $\mathrm{z}_{\mathrm{data},k}$ are evaluations of the solution $(v_{\text{\tiny FH}},w_{\text{\tiny FH}})$ of the FitzHugh-Nagumo model with $\mathrm{z}^{(0)}_{\mathrm{data},k}$ as initial condition:
\begin{linenomath}\begin{equation}
\left\{ \begin{array} {ll}
\displaystyle \dot{v}_{\text{\tiny FH}} + av_{\text{\tiny FH}}^3+bv_{\text{\tiny FH}}^2+cv_{\text{\tiny FH}} +dw_{\text{\tiny FH}} = f_v & \text{in }  (0,T),\\[10pt]
\displaystyle \dot{w}_{\text{\tiny FH}} + \eta w_{\text{\tiny FH}} + \gamma v_{\text{\tiny FH}} = f_w & \text{in } (0,T),\\[10pt]
(v_{\text{\tiny FH}},w_{\text{\tiny FH}})(0) = \mathrm{z}^{(0)}_{\mathrm{data},k} \in \R^2. &
 \end{array} \right. \label{ODEFH}
\end{equation}\end{linenomath}
For numerical realizations, we choose
\begin{linenomath}\begin{equation*}
a = 1/3, \quad b = 0, \quad c = -1.0, \quad d = 1.0, \quad \eta = 0.064, \quad \gamma = -0.08,
\quad f_v = 0.5, \quad f_w = 0.056.
\end{equation*}\end{linenomath}
We consider $K=7$ data, corresponding to the following initial conditions:
\begin{linenomath}\begin{equation}
\begin{array} {l}
\mathrm{z}^{(0)}_{\mathrm{data},1} = (0,0), \quad
\mathrm{z}^{(0)}_{\mathrm{data},2} = (1,1), \quad
\mathrm{z}^{(0)}_{\mathrm{data},3} = (-1,-1), \quad
\mathrm{z}^{(0)}_{\mathrm{data},4} = (1,0), \\
\mathrm{z}^{(0)}_{\mathrm{data},5} = (0,1), \quad
\mathrm{z}^{(0)}_{\mathrm{data},6} = (-1,1), \quad
\mathrm{z}^{(0)}_{\mathrm{data},7} = (1,-1).
\end{array} \label{super-data}
\end{equation}\end{linenomath}
System~\eqref{ODE} is discretized with the Crank-Nicholson method, and nonlinearities are treated with the Newton method. The adjoint system is discretized with an implicit Euler scheme. We discretize the optimality conditions corresponding to the following optimal control problem:
\begin{linenomath}\begin{equation} \label{pbODE} \tag{$\mathcal{P}_{\text{\tiny FH}}$}
\left\{ \begin{array} {l}
\displaystyle \min_{\mathbf{W} \in \mathcal{W}} \frac{1}{2} \sum_{k=1}^K \int_0^T |\mathrm{z}_k - \mathrm{z}_{\mathrm{data},k} |_{\R^2}^2 \d t +  \alpha \|\mathbf{W} \|_{\mathcal{W}}^2 , \\[15pt]
\text{subject to $(\mathrm{z}_k = (v,w),\bW)$ satisfying~\eqref{ODE} with $(v_0,w_0) = \mathrm{z}^{(0)}_{\mathrm{data},k}$.}
\end{array} \right.
\end{equation}\end{linenomath}
In~\eqref{ODE} the nonlinearity $\Phi = (\phi_v,\phi_w)$, given in~\eqref{sysmain2}, corresponds to a feedforward residual neural network of $L=7$ layers, parameterized with the weights $\mathbf{W}$, choosing $n_{\ell} = 2$ for all $1 \leq \ell \leq L$. The dimensions of the hidden layers are equal to $2$. The activation functions are smoothed ReLU, like those given in~\eqref{smoothReLU} with $\varepsilon = 2.0$. Choosing $\alpha = 0.01$, problem~\eqref{pbODE} is solved with the Barzilai-Borwein algorithm~\cite{BarBor} initialized with random weights and an Armijo rule. We illustrate the outcomes of the neural network $\Phi$ so computed in Figure~\ref{fig-num}, by testing the initial condition $(v_0,w_0) = (2,0)$, and comparing the corresponding solution with the one from the FitzHugh-Nagumo model.

\hspace{-5pt}\begin{minipage}{0.96\linewidth}
\begin{minipage}{0.96\linewidth}
\begin{figure}[H]
\includegraphics[trim = 0cm 0cm 0cm 0cm, clip, scale=0.32]{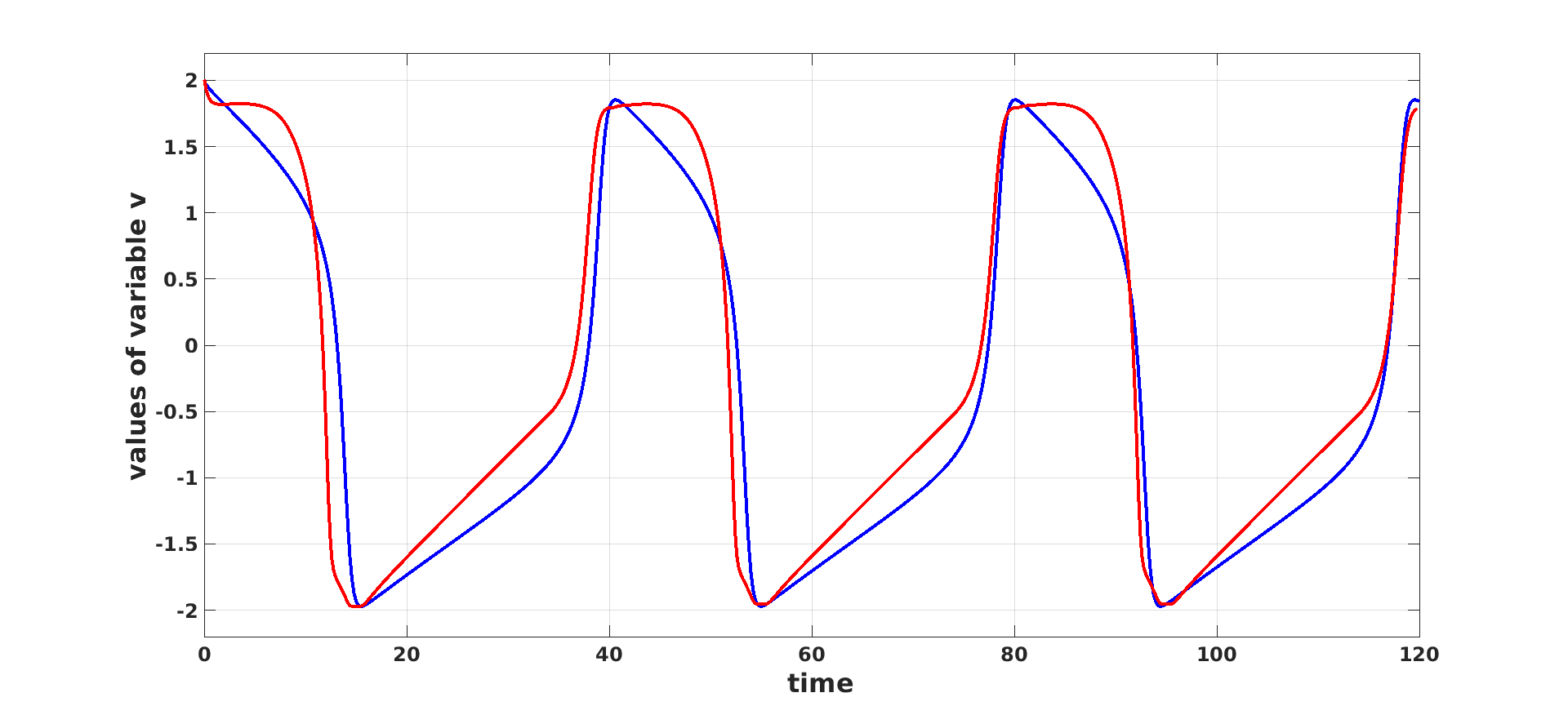}
\end{figure}
\end{minipage}
\vspace*{-20pt}
\begin{figure}[H]
\caption{With $(v_0,w_0) = (v_{\text{\tiny FH}},w_{\text{\tiny FH}})(0) = (2,0)$: The solution $v_{\text{\tiny FH}}$ of the FitzHugh-Nagumo model~\eqref{ODEFH} in blue, and in red the solution $v$ of~\eqref{ODE} with the neural network as $\Phi= (\phi_v,\phi_w)$. (For interpretation of the references to color in this figure legend, the reader is invited to refer to the digital version of this article.)\label{fig-num}}
\end{figure}
\end{minipage}\\
\FloatBarrier

We observe that minima and maxima of $v_{\text{\tiny FH}}$ are approximated very accurately, whereas the overall graph of $v_{\text{\tiny FH}}$ is roughly approximated (the misfit between the two curves is about 23.6\%). Likely this is due to the specific architecture of the neural network. We noted that reducing the value of the regularizing parameter $\varepsilon$ does not significantly change the results. However, starting with a large $\varepsilon$ can facilitate the convergence of gradient algorithms.

\subsection{Approximating the Aliev-Panfilov model}
\textcolor{black}{While the FitzHugh-Nagumo is of polynomial type, the Aliev-Panfilov model~\cite{Aliev1996} is {\it a priori} more complex and presents specific phenomena like {\it spiral} and {\it scroll} waves that have been studied in~\cite{Nash2004, Goektepe2010, Pagani2021}. The ODE version of this model writes as follows:
\begin{linenomath}
\begin{equation}
\left\{ \begin{array} {ll}
\displaystyle \dot{v}_{\text{\tiny AP}} + K_v v_{\text{\tiny AP}}(v_{\text{\tiny AP}}-a)(v_{\text{\tiny AP}}-1) + v_{\text{\tiny AP}}w_{\text{\tiny AP}} = f_v & \text{in }  (0,T),\\[10pt]
\displaystyle \dot{w}_{\text{\tiny AP}} +
\left( \varepsilon_0 + \frac{c_1w_{\text{\tiny AP}}}{c_2+u_{\text{\tiny AP}}} \right)
\left(w_{\text{\tiny AP}} + K_wv_{\text{\tiny AP}}(v_{\text{\tiny AP}}-b-1) \right)
= f_w & \text{in } (0,T),\\[10pt]
(v_{\text{\tiny AP}},w_{\text{\tiny AP}})(0) = \mathrm{z}^{(0)}_{\mathrm{data},k} \in \R^2. &
\end{array} \right. \label{ODEAP}
\end{equation}
\end{linenomath}
In the numerical illustrations we choose}

\textcolor{black}{\begin{linenomath}\begin{equation*}
K_v = 8.0, \quad K_w = 6.5, \quad a =b = 0.15, \quad
c_1 = 0.2, \quad c_2 = 0.3, \quad \varepsilon_0 = 0.002.
\end{equation*}\end{linenomath}
Like in the previous subsection, a neural network is trained by considering the data that correspond to the initial conditions given in~\eqref{super-data}. We use two types of activation functions: The so-called {\it Growing Cosine Unit}~(GCU):}

\textcolor{black}{\begin{linenomath}\begin{equation*}
\rho(x) = x\cos(x),
\end{equation*}\end{linenomath}
and the  {\it Hyperbolic tangent sigmoid}~(Tansig), which is simply the hyperbolic tangent function}

\textcolor{black}{\begin{linenomath}
\begin{equation*}
\rho(x) = \tanh(x).
\end{equation*}\end{linenomath}
The model~\eqref{ODE} so designed is tested when simulating a periodic traveling wave solution of the Aliev-Panfilov model with $(v_{\text{\tiny AP}}(0), w_{\text{\tiny AP}}(0)) = (0.75, 0.75)$ (see~\cite[Fig.~3]{Gani2016}). The comparison between the solutions of the ground truth ODE~\eqref{ODEAP} and the solutions obtained by solving~\eqref{ODE} with the trained neural networks as nonlinearity is presented in Figure~\ref{fig-AP}, for different neural network architectures. The latter are described in Table~\ref{table-arch}.}

\begin{table}[H]
\begin{center}
\begin{eqnarray*}
\begin{array} {|c|c|c|c|}
	\hline
	 \text{Architecture} &  \text{act. function} & \text{nb of layers} & \text{width of hidden layers}\\
	\hline
	 1 &  \text{GCU} & 7 & 2 \\
	\hline
	 2 &  \text{Tanh} & 5 & 8 \\
	\hline
\end{array} &
\end{eqnarray*}
\vspace*{-10pt}
		\caption{\textcolor{black}{Description of the two types of architectures.}}
		\label{table-arch}
\end{center}
\end{table}
\FloatBarrier

\hspace{-5pt}\begin{minipage}{0.96\linewidth}
\begin{tabular} {c|c}
\begin{minipage}{0.47\linewidth}
\begin{figure}[H]
\includegraphics[trim = 0cm 0cm 0cm 1.5cm, clip, scale=0.32]{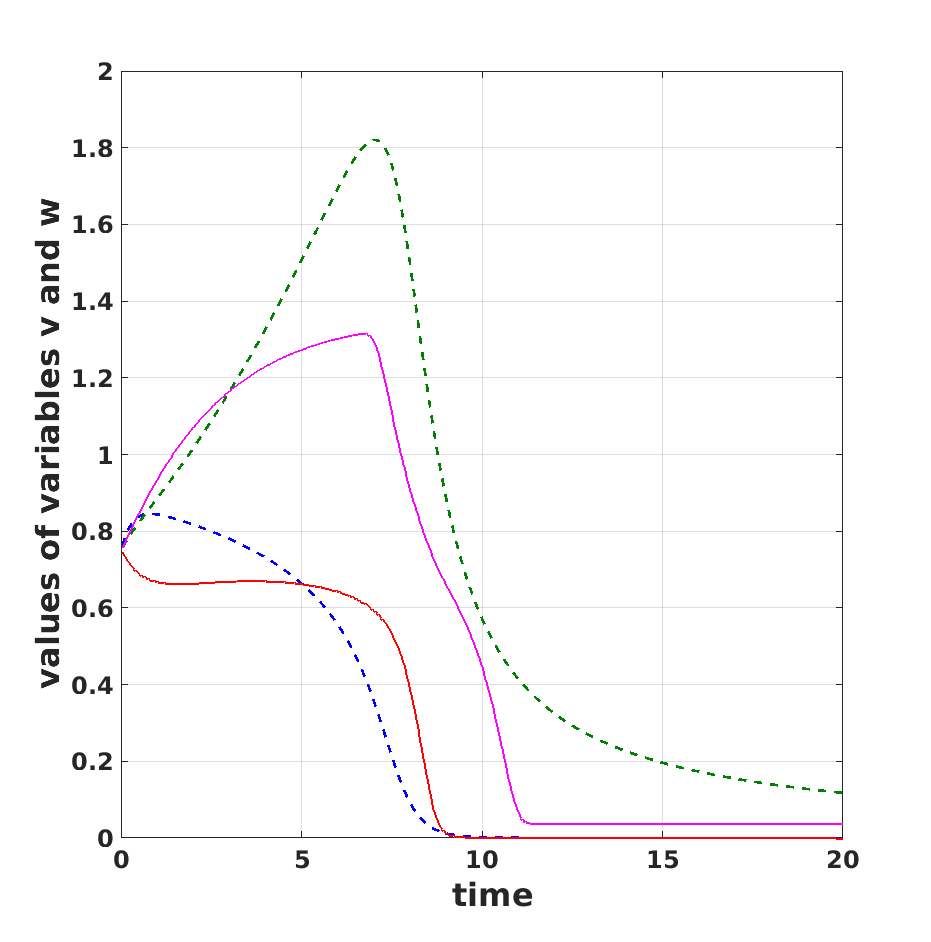}
\end{figure}
\end{minipage}
&
\begin{minipage}{0.47\linewidth}
\begin{figure}[H]
\includegraphics[trim = 0cm 0cm 0cm 1.5cm, clip, scale=0.32]{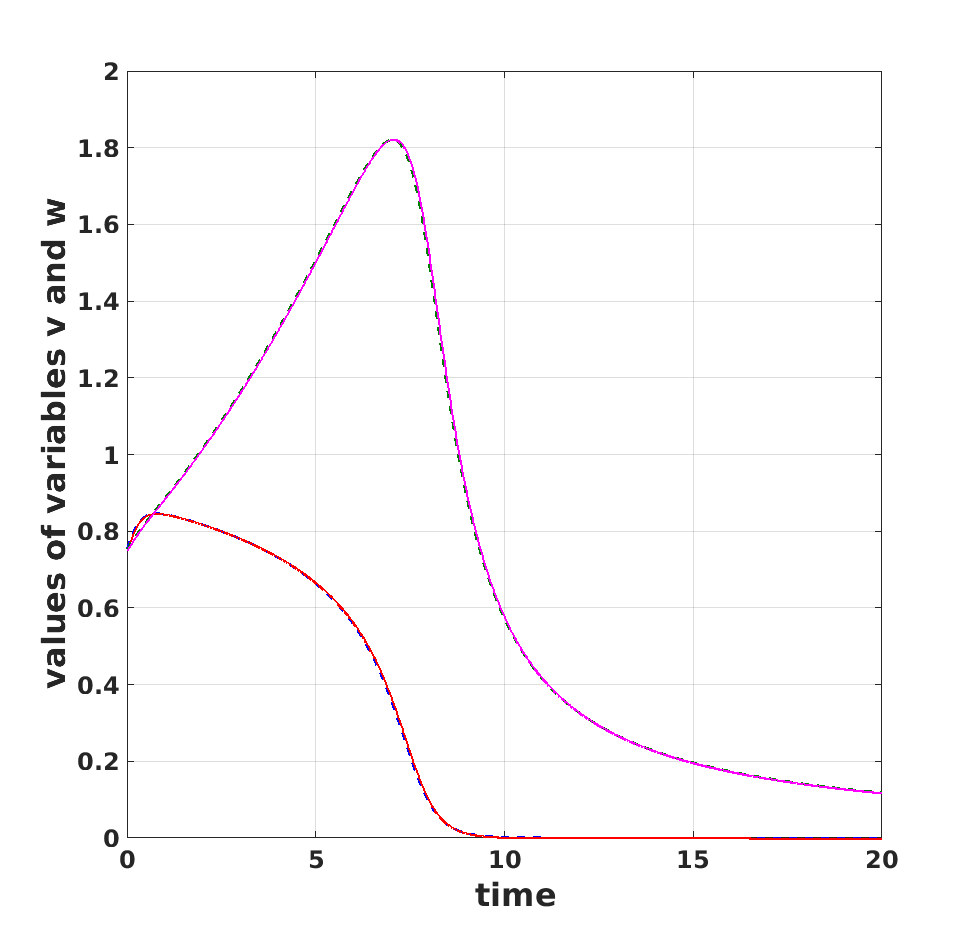}
\end{figure}
\end{minipage}
\end{tabular}
\vspace*{-20pt}
\begin{figure}[H]
\caption{\textcolor{black}{With $(v_0,w_0) = (v_{\text{\tiny AP}},w_{\text{\tiny AP}})(0) = (0.75,0.75)$: The solution $(v_{\text{\tiny AP}},w_{\text{\tiny AP}})$ of the Aliev-Panfilov model~\eqref{ODEAP} in dashed blue and green respectively, in red and magenta respectively, the solution $(v,w)$ of~\eqref{ODE} with the trained neural networks as $\Phi= (\phi_v,\phi_w)$. Left: with Architecture~1. Right: with Architecture~2.}\label{fig-AP}}
\end{figure}
\end{minipage}\\
\FloatBarrier

\textcolor{black}{Figure~\ref{fig-AP} shows that the second architecture enables us to obtain a very accurate neural network (the corresponding misfit is about $0,77 \%$), while the first architecture introduces a non-negligible misfit, even if the qualitative behavior seems to be consistent.  We have also tested the activation function $\tanh$ with Architecture~1, namely~7 layers and~2 neurons per layer, as in GCU, and failed. With $\tanh$ as activation function, it is important to increase the number of hidden layer, and while  the number of layers can be reduced (compared to GCU as activation function), to obtain excellent results.}

\textcolor{black}{Next, these neural networks are re-used for performing in 1D the numerical simulation of the PDE version of this trained model, with $\Omega = (0,10)$ and $T=10$, namely:
\begin{linenomath}\begin{equation}
\left\{ \begin{array} {ll}
\displaystyle \dot{v} -D_v \Delta v + \phi_v(v,w) = f_v &
\text{in } \Omega \times(0,T),\\[10pt]
\displaystyle \dot{w} -D_w \Delta w + \phi_w(v,w) = f_w &
\text{in } \Omega \times (0,T),\\[5pt]
\displaystyle \frac{\p v}{\p n} = 0 & \text{on } \p \Omega \times (0,T), \\[5pt]
(v,w)(0) = (v_0,w_0)\in \R^2. &
\end{array} \right. \label{PDEAP}
\end{equation}\end{linenomath}
For this purpose, we implemented in C++, using the Getfem++ Library~\cite{Getfem}, a solver for PDE~\eqref{PDEAP}. The visualization of the solutions is done with Matlab. The space discretization is made with P1 elements and the time discretization with an implicit Euler scheme. We chose $D_v = 0.0005$, $D_w = 0.005$, and the right-and-sides $f_v$ and $f_w$ were chosen such that the ground-truth solutions of the Aliev-Panfilov model correspond to
\begin{linenomath}\begin{equation}
\begin{array} {l}
\displaystyle  v_{\text{\tiny AP}}(x,t) = \sin(2(x-ct)) + \sin(4(x-ct)), \quad
w_{\text{\tiny AP}}(x,t) = \cos(2(x-ct)) +
\frac{1}{2} \cos(4(x-ct)),\\[5pt]
\displaystyle  f_v = \dot{v}_{\text{\tiny AP}} - D_v\Delta v_{\text{\tiny AP}} +
K_v v_{\text{\tiny AP}}(v_{\text{\tiny AP}}-a)(v_{\text{\tiny AP}}-1) + v_{\text{\tiny AP}}w_{\text{\tiny AP}}, \\[5pt]
\displaystyle f_w = \dot{w}_{\text{\tiny AP}} - D_w\Delta w_{\text{\tiny AP}}
+\left( \varepsilon_0 + \frac{c_1w_{\text{\tiny AP}}}{c_2+u_{\text{\tiny AP}}} \right)
\left(w_{\text{\tiny AP}} + K_wv_{\text{\tiny AP}}(v_{\text{\tiny AP}}-b-1) \right),\\
(v_0,w_0) = \big(v_{\text{\tiny AP}}(\cdot,0), w_{\text{\tiny AP}}(\cdot,0)\big),
\end{array} \label{super-rhs}
\end{equation}\end{linenomath}
with velocity $c = 0.1$. Note that these forced solutions are traveling waves. The solutions obtained by solving~\eqref{PDEAP} with the two different neural networks are represented in Figures~\ref{fig-PDE} and~\ref{fig-PDE-OK}. We obtain traveling waves, in spite of the fact that we did not impose this periodicity in the model, except implicitly in the right-hand-sides~$f_v$ and~$f_w$.}\\
\begin{minipage}{0.96\linewidth}
\begin{tabular} {c|c}
\begin{minipage}{0.45\linewidth}
\begin{figure}[H]
\includegraphics[trim = 0cm 0cm 0cm 1.5cm, clip, scale=0.32]{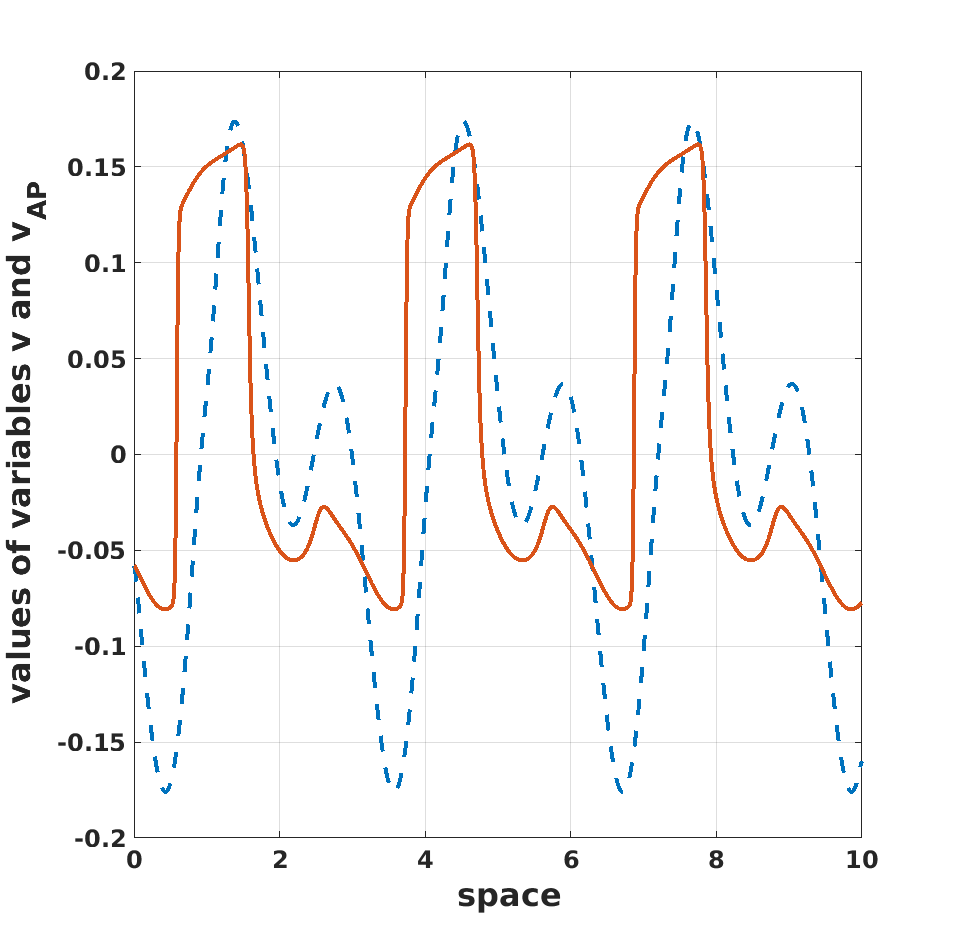}
\end{figure}
\end{minipage}
&
\begin{minipage}{0.45\linewidth}
\begin{figure}[H]
\includegraphics[trim = 0cm 0cm 0cm 1.5cm, clip, scale=0.32]{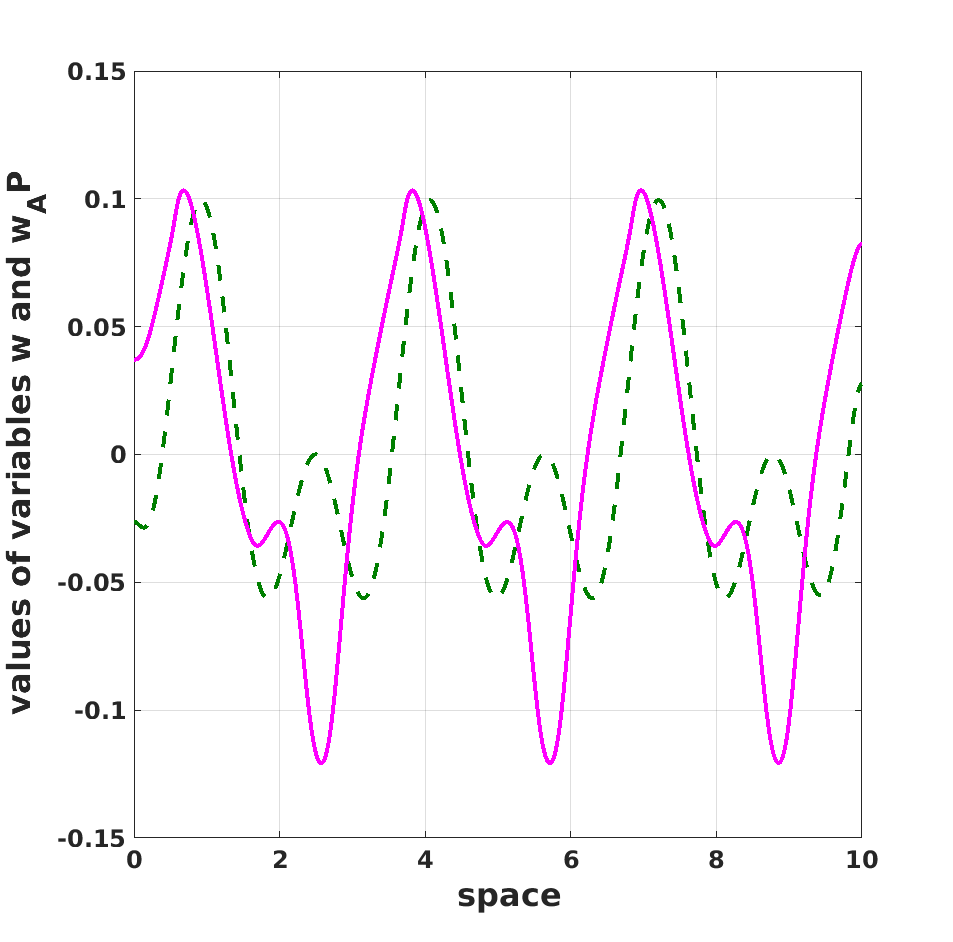}
\end{figure}
\end{minipage}
\end{tabular}
\vspace*{-20pt}
\begin{figure}[H]
\caption{\textcolor{black}{At time $t = 9.00$, the computed solution $(v_{\text{\tiny AP}},w_{\text{\tiny AP}})$ of the Aliev-Panfilov PDE model corresponding to~\eqref{super-rhs} represented in dashed blue and green respectively, compared with the computed solution $(v,w)$ of~\eqref{PDEAP} with the neural network as $\Phi= (\phi_v,\phi_w)$, represented in red and magenta, respectively. Here the neural network was trained with Architecture~1.}\label{fig-PDE}}
\end{figure}
\end{minipage}\\
\begin{minipage}{0.96\linewidth}
\begin{tabular} {c|c}
\begin{minipage}{0.45\linewidth}
\begin{figure}[H]
\includegraphics[trim = 0cm 0cm 0cm 1.5cm, clip, scale=0.32]{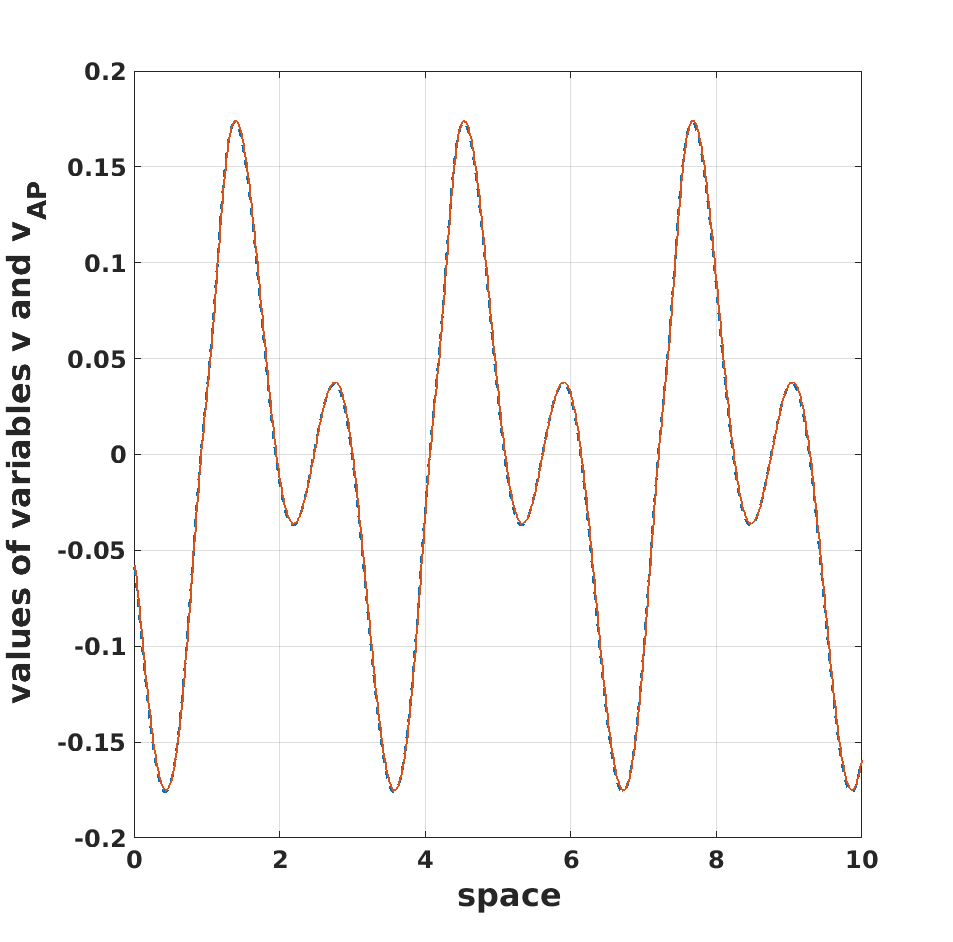}
\end{figure}
\end{minipage}
&
\begin{minipage}{0.45\linewidth}
\begin{figure}[H]
\includegraphics[trim = 0cm 0cm 0cm 1.5cm, clip, scale=0.32]{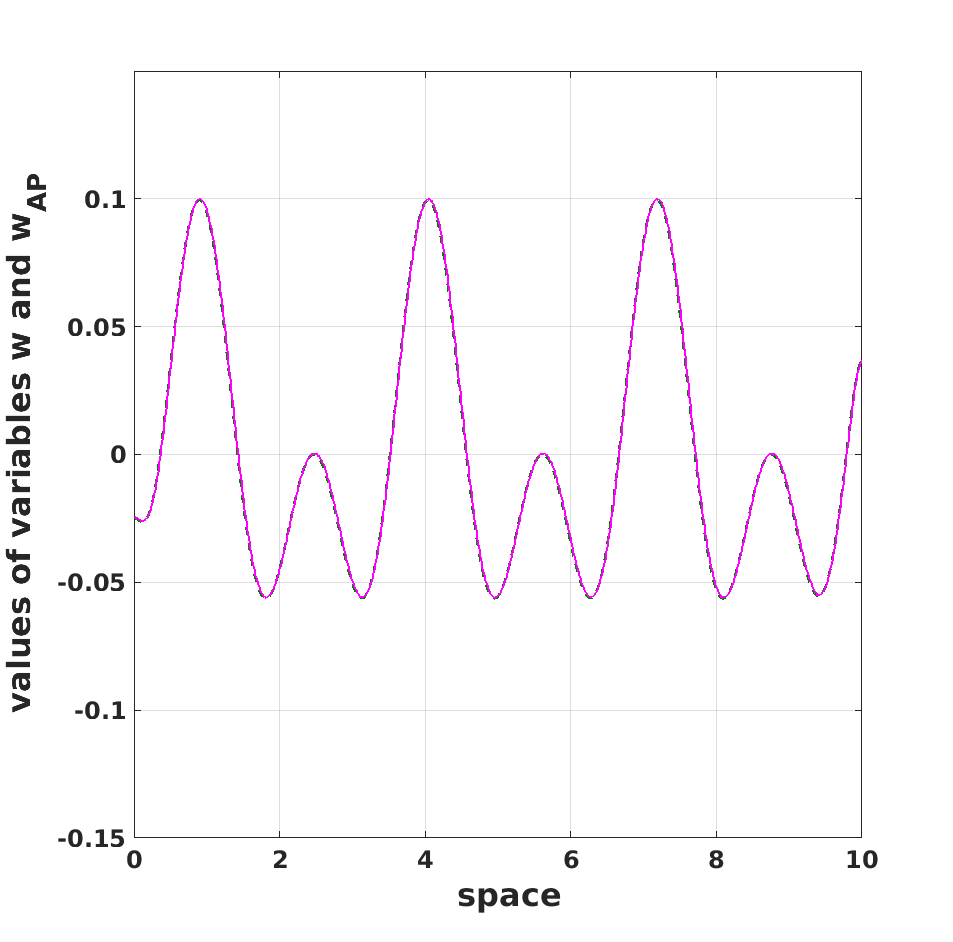}
\end{figure}
\end{minipage}
\end{tabular}
\vspace*{-20pt}
\begin{figure}[H]
\caption{\textcolor{black}{Same as Figure~\ref{fig-PDE}, but with the neural network trained with Architecture~2 as~$\Phi= (\phi_v,\phi_w)$.
}\label{fig-PDE-OK}}
\end{figure}
\end{minipage}\\

\textcolor{black}{In Figures~\ref{fig-PDE} and~\ref{fig-PDE-OK}, the profiles of the traveling waves are depicted at $t=9$. Figure~\ref{fig-PDE} corresponds to the first architecture, and Figure~\ref{fig-PDE-OK} to the second one. As in the case of the ODE version, also for the PDE model~\eqref{PDEAP}, Architecture~2 is superior to Architecture~1 for the reconstruction of the nonlinearity.}

\section{Conclusion} \label{sec-conc}

We addressed a model identification problem via optimal control techniques, with parameters defining neural networks. The choice of artificial neural networks for determining the nonlinearity of the model constitutes a change in the nature of the approximation compared to SINDy for example. The lack of differentiability in cases where non-smooth activation functions are considered requires regularization techniques in order to derive rigorously optimality conditions. \textcolor{black}{A further approach would consist in deriving necessary optimality conditions directly, without regularization. For that purpose sub-differential calculus techniques could be applied, as such methods were recently deployed in an infinite dimensional setting~\cite{Frankowska2018}. In the numerical realizations we observed that the choice of the activation functions as well as the width of the hidden layers can have a significant influence on the quality of the obtained reconstruction.} 

{\footnotesize \printbibliography}

\end{document}